\documentclass[11pt]{amsart}

\usepackage{a4wide, amsmath, amsfonts, amssymb, mathrsfs, amsthm, breqn, enumerate}
\usepackage[hyperfootnotes=false]{hyperref}

\allowdisplaybreaks[2]

\numberwithin{equation}{section}

\newtheorem{theorem}{Theorem}[section]
\newtheorem{lemma}[theorem]{Lemma}
\newtheorem{corollary}[theorem]{Corollary}
\newtheorem{remark}[theorem]{Remark}
\newtheorem{proposition}[theorem]{Proposition}
\newtheorem{definition}[theorem]{Definition}
\newtheorem{example}[theorem]{Example}

\newtheorem*{property}{Property}

\newcommand{\dd}{\,\mathrm{d}}
\renewcommand{\d}{\mathrm{d}}
\newcommand{\D}{\mathrm{D}}
\renewcommand{\epsilon}{\varepsilon}
\renewcommand{\phi}{\varphi}
\newcommand{\R}{\mathbb{R}}
\newcommand{\N}{\mathbb{N}}

\newcommand{\X}{\mathbb{X}}
\renewcommand{\P}{\mathbb{P}}
\newcommand{\1}{\mathbf{1}}
\newcommand{\bS}{\mathbf{S}}
\newcommand{\bX}{\mathbf{X}}

\newcommand{\tX}{\tilde{X}}

\newcommand{\tZ}{\tilde{Z}}
\newcommand{\tbbX}{\tilde{\X}}
\newcommand{\tbX}{\tilde{\bX}}
\newcommand{\rp}{\mathscr{V}^p}
\newcommand{\crpZ}{\mathcal{V}^{q,r}_Z}
\newcommand{\crpX}{\mathcal{V}^{q,r}_X}
\newcommand{\crptZ}{\mathcal{V}^{q,r}_{\tilde{Z}}}
\newcommand{\crptX}{\mathcal{V}^{q,r}_{\tilde{X}}}
\newcommand{\cP}{\mathcal{P}}
\newcommand{\tf}{\tilde{f}}
\newcommand{\cV}{\mathcal{V}}
\newcommand{\tphi}{\tilde{\phi}}

\newcommand{\cL}{\mathcal{L}}
\newcommand{\tS}{\tilde{S}}
\newcommand{\tbS}{\tilde{\mathbf{S}}}
\newcommand{\cG}{\mathcal{G}}
\newcommand{\tF}{\tilde{F}}
\newcommand{\tG}{\tilde{G}}

\newcommand{\crpmuq}{\mathcal{V}^{q,r}_\mu}

\title[A C{\`a}dl{\`a}g Rough Path Foundation for Robust Finance]{A C{\`a}dl{\`a}g Rough Path Foundation\\
for Robust Finance}

\author[Allan]{Andrew L. Allan}
\address{Andrew L. Allan, Durham University, United Kingdom}
\email{andrew.l.allan@durham.ac.uk}

\author[Liu]{Chong Liu}
\address{Chong Liu, ShanghaiTech University, China}
\email{liuchong@shanghaitech.edu.cn}

\author[Pr{\"o}mel]{David J. Pr{\"o}mel}
\address{David J. Pr{\"o}mel, University of Mannheim, Germany}
\email{proemel@uni-mannheim.de}

\date{\today}

\begin{document}

\begin{abstract}
  Using rough path theory, we provide a pathwise foundation for stochastic It{\^o} integration, which covers most commonly applied trading strategies and mathematical models of financial markets, including those under Knightian uncertainty. To this end, we introduce the so-called Property (RIE) for c{\`a}dl{\`a}g paths, which is shown to imply the existence of a c{\`a}dl{\`a}g rough path and of quadratic variation in the sense of F{\"o}llmer. We prove that the corresponding rough integrals exist as limits of left-point Riemann sums along a suitable sequence of partitions. This allows one to treat integrands of non-gradient type, and gives access to the powerful stability estimates of rough path theory. Additionally, we verify that (path-dependent) functionally generated trading strategies and Cover's universal portfolio are admissible integrands, and that Property (RIE) is satisfied by both (Young) semimartingales and typical price paths.
\end{abstract}

\maketitle





\bigskip

\section{Introduction}

A fundamental pillar of mathematical finance is the theory of stochastic integration initiated by K.~It{\^o} in the 1940s. It{\^o}'s stochastic integration not only allows for a well-posedness theory for most probabilistic models of financial markets, but also comes with invaluable properties, such as having an integration by parts formula and chain rule, and that of being a continuous operator (with respect to suitable spaces of random variables), which is essential for virtually all applications. However, despite the elegance and success of It{\^o} integration, it also admits some significant drawbacks from both theoretical and practical perspectives.

The construction of the It{\^o} integral requires one to fix a probability measure a priori, and is usually based on a limiting procedure of approximating Riemann sums in probability. While in mathematical finance the It{\^o} integral usually represents the capital gain process from continuous-time trading in a financial market, it lacks a robust pathwise meaning. That is, the stochastic It{\^o} integral does not have a well-defined value on a given ``state of the world'', e.g.~a realized price trajectory of a liquidly traded asset on a stock exchange. This presents a gap between probabilistic models and their financial interpretation. Addressing the pathwise meaning of stochastic integration has led to a stream of literature beginning with the classical works of Bichteler \cite{Bichteler1981} and Willinger and Taqqu \cite{Willinger1989}; see also Karandikar \cite{Karandikar1995} and Nutz \cite{Nutz2012}.

The requirement of fixing a probability measure to have access to It{\^o} integration becomes an even more severe obstacle when one wants to develop mathematical finance under model risk---also known as Knightian uncertainty. Starting from the seminal works of Avellaneda, Levy and Par\'as \cite{Avellaneda1995} and Lyons \cite{Lyons1995}, there has been an enormous and on-going effort to treat the challenges posed by model risk in mathematical finance, that is, the risk stemming from the possible misspecification of an adopted stochastic model, typically represented by a single fixed probability measure. The majority of the existing robust treatments of financial modelling replace the single probability measure by a family of (potentially singular) probability measures, or even take so-called model-free approaches, whereby no probabilistic structure of the underlying price trajectories is assumed; see for example Hobson \cite{Hobson2003} for classical lecture notes on robust finance. In particular, the latter model-free approaches often require a purely deterministic integration theory sophisticated enough to handle the irregular sample paths of standard continuous-time financial models and commonly employed functionally generated trading strategies.

In the seminal paper~\cite{Follmer1981}, F{\"o}llmer provided the first deterministic analogue to stochastic It{\^o} integration which had the desired properties required by financial applications. Indeed, assuming that a c{\`a}dl{\`a}g path $S \colon [0,T] \to \R^d$ possesses a suitable notion of quadratic variation along a sequence $(\mathcal{P}^n)_{n \in \N}$ of partitions of the interval $[0,T]$, F{\"o}llmer proved that the limit
\begin{equation*}
\int_0^t \D f(S_u)\dd S_u := \lim_{n\to \infty} \sum_{[u,v]\in \mathcal{P}^n} \D f(S_u) (S_{v\wedge t}-S_{u\wedge t}), \qquad t \in [0,T],
\end{equation*}
where $\D f$ denotes the gradient of $f$, exists for all twice continuously differentiable functions $f \colon \R^d \to \R$. The resulting pathwise integral $\int_0^t \D f(S_u)\dd S_u$ is often called the F{\"o}llmer integral, and has proved to be a valuable tool in various applications in model-free finance; for some recent examples we refer to F{\"o}llmer and Schied \cite{Follmer2013}, Davis, Ob\l{}\'oj and Raval \cite{Davis2014}, Schied, Speiser and Voloshchenko \cite{Schied2018} and Cuchiero, Schachermayer and Wong \cite{Cuchiero2019}. In fact, even classical Riemann--Stieltjes integration has been successfully used as a substitution to It{\^o} integration in model-free finance; see e.g.~Dolinsky and Soner \cite{Dolinsky2014} or Hou and Ob\l{}\'oj \cite{Hou2018}.

By now arguably the most general pathwise (stochastic) integration theory is provided by the theory of rough paths, as introduced by Lyons \cite{Lyons1998}, and its recent extension to c{\`a}dl{\`a}g rough paths by Friz and Shekhar \cite{Friz2017}, Friz and Zhang \cite{Friz2018} and Chevyrev and Friz \cite{Chevyrev2019}. Rough integration can be viewed as a generalization of Young integration which is able to handle paths of lower regularity. While rough integration allows one to treat the sample paths of numerous stochastic processes as integrators and offers powerful pathwise stability estimates, it comes with a pitfall from a financial perspective: the rough integral is defined as a limit of so-called compensated Riemann sums, and thus apparently does not correspond to the canonical financial interpretation as the capital gain process generated by continuous-time trading. Even worse, choosing a rough path without care might lead to an anticipating integral, corresponding e.g.~to Stratonovich integration, thus introducing undesired arbitrage when used as a capital process.

We overcome these issues by introducing the so-called Property (RIE) for a c{\`a}dl{\`a}g path $S \colon [0,T] \to \R^d$ and a sequence $(\mathcal{P}^n)_{n \in \N}$ of partitions of the interval $[0,T]$. This property is very much in the same spirit as F{\"o}llmer's assumption of quadratic variation along a sequence of partitions. Indeed, we show that Property (RIE) implies the existence of quadratic variation in the sense of F{\"o}llmer, and even the existence of a c{\`a}dl{\`a}g rough path $\mathbf{S}$ above $S$, which, loosely speaking, corresponds to an ``It\^o'' rough path in a probabilistic setting. Assuming Property (RIE), we prove that the corresponding rough integrals exist as limits of left-point Riemann sums along the sequence of partitions $(\mathcal{P}^n)_{n \in \N}$. This result restores the canonical financial interpretation for rough integration, and links it to F{\"o}llmer integration for c{\`a}dl{\`a}g paths. Property (RIE) was previously introduced by Perkowski and Pr\"omel \cite{Perkowski2016} for continuous paths, though we emphasize that the present more general c{\`a}dl{\`a}g setting requires quite different techniques compared to the continuous setting of \cite{Perkowski2016}.

Given the aforementioned results, a c{\`a}dl{\`a}g path which satisfies Property (RIE) permits the path-by-path existence of rough integrals with their desired financial interpretation, and moreover maintains access to their powerful stability results which ensure that the integral is a continuous operator. This appears to be a significant advantage compared to the classical notions of pathwise stochastic integration in \cite{Bichteler1981,Willinger1989,Karandikar1995,Nutz2012}, which do not come with such stability estimates. In particular, the pathwise stability results of rough path theory allow one to prove a model-free version of the so-called fundamental theorem of derivative trading---see Armstrong, Bellani, Brigo and Cass \cite{Armstrong2018}---and may be of interest when investigating discretization errors of continuous-time trading in model-free finance; see Riga \cite{Riga2016}. Furthermore, in contrast to F{\"o}llmer integration, rough integration allows one to consider general functionally generated integrands $g(S_t)$, where $g$ is a general (sufficiently smooth) function $g \colon \R^d \to \R^d$, and \emph{not} necessarily the gradient of another vector field $f \colon \R^d \to \R$. For instance, model-free portfolio theory constitutes a research direction in which it is beneficial to consider non-gradient trading strategies; see Allan, Cuchiero, Liu and Pr\"omel \cite{Allan2023}. Even more generally, rough integration allows one to treat path-dependent functionally generated options in the sense of Dupire \cite{Dupire2019}, and pathwise versions of Cover's universal portfolio, as discussed in Section~\ref{sec:trading strategies}.

Of course, it remains to verify that Property (RIE) is a reasonable modelling assumption in mathematical finance, in the sense that it is fulfilled almost surely by sample paths of the commonly used probabilistic models of financial markets. Since it seems natural that continuous-time trading takes place when the underlying price process fluctuates, we employ sequences of partitions based on such a ``space discretization''. For such sequences of partitions, we show that the sample paths of c{\`a}dl{\`a}g semimartingales almost surely satisfy Property (RIE). This result is then extended to so-called Young semimartingales, which are stochastic possesses given by the sum of a c{\`a}dl{\`a}g local martingale and an adapted c{\`a}dl{\`a}g process of finite $q$-variation for some $q < 2$. Finally, we prove that Property (RIE) is satisfied by typical price paths in the sense of Vovk~\cite{Vovk2008}, which correspond to a model-free version of ``no unbounded profit with bounded risk''.

\medskip

\noindent \textbf{Organization of the paper:} In Section~\ref{sec:rough path integration} we introduce Property~\textup{(RIE)} and verify the properties of the associated rough integration as described above. In Section~\ref{sec:trading strategies} we exhibit functionally generated trading strategies and generalizations thereof which provide valid integrands for rough integration. In Section~\ref{sec:stochastic processes} we prove that (Young) semimartingales and typical price paths satisfy Property~\textup{(RIE)}.

\medskip

\noindent\textbf{Acknowledgments:} A.~L.~Allan gratefully acknowledges financial support by the Swiss National Science Foundation via Project 200021\textunderscore 184647. C.~Liu gratefully acknowledges support from the Early Postdoc.Mobility Fellowship (No.~P2EZP2\textunderscore 188068) of the Swiss National Science Foundation, and from the G.~H.~Hardy Junior Research Fellowship in Mathematics awarded by New College, Oxford.

\section{Rough integration under Property~(RIE)}\label{sec:rough path integration}

In this section we develop pathwise integration under Property~\textup{(RIE)}. We set up the essential ingredients from rough path theory in Section~\ref{subsec:cadlag rough path} and show in Section~\ref{subsec:integration under RIE} that paths satisfying~\textup{(RIE)} serve as suitable integrators in mathematical finance. Finally, in Section~\ref{subsec:Follmer integration} we connect Property~\textup{(RIE)} with the existence of quadratic variation in the sense of F{\"o}llmer.

\subsection{Basic notation}

Let $(\R^d,|\hspace{1pt}\cdot\hspace{1pt}|)$ denote standard Euclidean space, and let $D([0,T];\R^d)$ denote the space of all c{\`a}dl{\`a}g (i.e.~right-continuous with left-limits) functions from $[0,T] \to \R^d$. A partition~$\mathcal{P}=\mathcal{P}([s,t])$ of the interval~$[s,t]$ is a set of essentially disjoint intervals covering $[s,t]$, i.e.~$\mathcal{P} = \{s = t_0 < t_1 < \cdots < t_{N} = t\}$ for some $N \in \N$. The mesh size of a partition $\mathcal{P}$ is given by $|\mathcal{P}| := \max\{|t_{i+1} - t_i| : i=0,\dots,N-1\}$, and, for a partition $\mathcal{P}$ of the interval $[s,t]$ and a subinterval $[u,v] \subset [s,t]$, we write $\mathcal{P}([u,v]) := (\cP \cup \{u,v\}) \cap [u,v] = \{r \in \cP \cup \{u,v\} : u \leq r \leq v\}$, for the restriction of the partition~$\mathcal{P}$ to the interval~$[u,v]$. A sequence $(\mathcal{P}^n)_{n\in \N}$ of partitions is called \emph{nested} if $\mathcal{P}^n \subset \mathcal{P}^{n+1}$ for all $n\in \N$.

Setting $\Delta_{[0,T]} := \{(s,t) \in [0,T]^2 : s \leq t\}$, a \emph{control function} is defined as a function $w \colon \Delta_{[0,T]} \to [0, \infty)$ which is superadditive, in the sense that $w(s,u) + w(u,t) \leq w(s,t)$ for all $0 \leq s \leq u \leq t \leq T$.

\smallskip

Throughout this section we fix a finite time interval $[0,T]$ and the dimension $d \in \N$. We also adopt the convention that, given a path $A$ defined on $[0,T]$, we will write $A_{s,t} := A_t - A_s$ for the increment of $A$ over the interval $[s,t]$. Note however that whenever $A$ is a two-parameter function defined on $\Delta_{[0,T]}$, then the notation $A_{s,t}$ will simply denote the value of $A$ evaluated at the pair of times $(s,t) \in \Delta_{[0,T]}$.

\smallskip

If $A$ denotes either a path from $[0,T] \to E$ or a two-parameter function from $\Delta_{[0,T]} \to E$ for some normed vector space $E$, then, for any $p \in [1,\infty)$, the $p$-variation of $A$ over the interval $[s,t]$ is defined by
\begin{equation*}
  \|A\|_{p,[s,t]} := \bigg(\sup_{\mathcal{P}([s,t])} \sum_{[u,v] \in \mathcal{P}([s,t])} |A_{u,v}|^p\bigg)^{\frac{1}{p}}
\end{equation*}
where the supremum is taken over all partitions $\mathcal{P}([s,t])$ of the interval $[s,t] \subseteq [0,T]$, and in the case when $A$ is a path we write $A_{u,v} := A_v - A_u$. If $\|A\|_{p,[0,T]} < \infty$ then $A$ is said to have finite $p$-variation.

We write $D^p = D^p([0,T];E)$ for the space of all c\`adl\`ag paths $A \colon [0,T] \to E$ of finite $p$-variation, and we similarly write $D^p_2 = D^p_2(\Delta_{[0,T]};E)$ for the space of two-parameter functions $A \colon \Delta_{[0,T]} \to E$ of finite $p$-variation which are such that the maps $s \mapsto A_{s,t}$ for fixed $t$, and $t \mapsto A_{s,t}$ for fixed $s$, are both c{\`a}dl{\`a}g. Note that $A$ having finite $p$-variation is equivalent to the existence of a control function $w$ such that $|A_{s,t}|^p \leq w(s,t)$ for all $(s,t) \in \Delta_{[0,T]}$. (For instance, one may take $w(s,t) = \|A\|_{p,[s,t]}^p$.)

\subsection{C{\`a}dl{\`a}g rough path theory and Property~(RIE)}\label{subsec:cadlag rough path}

While rough path theory has by now been well studied in the case of continuous paths, as exhibited in a number of books, notably Friz and Hairer \cite{FrizHairer2020}, its extension to c{\`a}dl{\`a}g paths appeared only recently, starting with Friz and Shekhar \cite{Friz2017}. In this section we mainly rely on results regarding forward integration with respect to c{\`a}dl{\`a}g rough paths as presented in Friz and Zhang \cite{Friz2018}.

\smallskip

In the following we fix $p \in (2,3)$ and $q \geq p$ such that
\begin{equation*}
  \frac{2}{p} + \frac{1}{q} > 1,
\end{equation*}
and define $r > 1$ by the relation
\begin{equation*}
  \frac{1}{r} = \frac{1}{p} + \frac{1}{q}.
\end{equation*}
This means in particular that $1 < p/2 \leq r < p \leq q < \infty$.

\smallskip

Throughout the paper, we will use the symbol $\lesssim$ to denote inequality up to a multiplicative constant which depends only on the numbers $p, q$ and $r$ as chosen above.

\smallskip

We begin by recalling the definition of a c{\`a}dl{\`a}g rough path, as well as the corresponding notion of controlled paths. In the following we will write $A \otimes B$ for the tensor product of two vectors $A, B \in \R^d$, i.e.~the $d \times d$-matrix with $(i,j)$-component given by $[A \otimes B]^{ij} = A^i B^j$ for $1 \leq i, j \leq d$.

\begin{definition}\label{def:rough path}
  We say that a triplet $\bX = (X,Z,\X)$ is a \emph{(c{\`a}dl{\`a}g) $p$-rough path} (over $\R^d$) if $X \in D^p([0,T];\R^d)$, $Z \in D^p([0,T];\R^d)$ and $\X \in D_2^{p/2}(\Delta_{[0,T]};\R^{d \times d})$, and if Chen's relation:
  \begin{equation}\label{eq:Chens relation}
    \X_{s,t} = \X_{s,u} + \X_{u,t} + Z_{s,u} \otimes X_{u,t} 
  \end{equation}
  holds for all times $0 \leq s \leq u \leq t \leq T$. We denote the space of c{\`a}dl{\`a}g rough paths by $\rp$.
\end{definition}

The unfamiliar reader is encouraged to check that, given c{\`a}dl{\`a}g paths $X$ and $Z$ of bounded variation, setting $\X_{s,t} = \int_s^t Z_{s,u} \otimes \d X_u \equiv \int_s^t Z_u \otimes \d X_u - Z_s \otimes X_{s,t}$ for $(s,t) \in \Delta_{[0,T]}$, with the integral defined as a limit of left-point Riemann sums, gives a $p$-rough path. Although the integral $\int_s^t Z_{s,u} \otimes \d X_u$ is not in general well-defined when $X$ and $Z$ are not of bounded variation, given a rough path $(X,Z,\X)$, we may think of $\X$ as postulating a ``candidate'' for the value of such integrals.

\begin{remark}
  The definition of rough paths we have introduced above looks slightly different to the standard definition, in which one takes $X = Z$. Our definition is slightly more general, but the corresponding theory works in exactly the same way, and turns out to be more convenient in the context of Property~\textup{(RIE)} as we will see later.

  More precisely, later the matrix $\X_{s,t}$ will for us represent the (a priori ill-defined) `integral' $\int_s^t S_{s,u} \otimes \d S_u$, which will be defined as the limit as $n \to \infty$ of the Riemann sums $(\int_s^t S^n_{s,u} \otimes \d S_u)_{n \in \N}$ appearing in Property~\textup{(RIE)} below. In the continuous (i.e.~without jumps) setting of Perkowski and Pr\"omel \cite{Perkowski2016}, a linear interpolation is used to provide a continuous approximation of $S^n$, leading to a Stratonovich type integral in the limit, which is subsequently converted back into an It\^o type integral. Thanks to the recently developed theory of c{\`a}dl{\`a}g rough paths, here we can use a more direct argument which avoids this detour. This means working directly with the integral $\int_s^t S^n_{s,u} \otimes \d S_u$, which corresponds to taking $X = S$ and $Z = S^n$ in Definition~\ref{def:rough path}, thus requiring $X \neq Z$.
\end{remark}

For two rough paths, $\bX = (X,Z,\X)$ and $\tbX = (\tX,\tZ,\tbbX)$, we use the seminorm
\begin{equation*}
  \|\bX\|_{p,[s,t]} := \|X\|_{p,[s,t]} + \|Z\|_{p,[s,t]} + \|\X\|_{\frac{p}{2},[s,t]},
\end{equation*}
and the pseudometric
\begin{equation*}
  \|\bX;\tbX\|_{p,[s,t]} := \|X - \tX\|_{p,[s,t]} + \|Z - \tZ\|_{p,[s,t]} + \|\X - \tbbX\|_{\frac{p}{2},[s,t]},
\end{equation*}
for $[s,t] \subseteq [0,T]$.

In the following we write $\cL(\R^d;\R^d)$ for the space of linear maps from $\R^d \to \R^d$.

\begin{definition}\label{def: controlled path}
  Let $Z \in D^p([0,T];\R^d)$. We say that a pair $(F,F')$ is a \emph{controlled path} (with respect to $Z$), if $F \in D^{p}([0,T];\R^d)$, $F' \in D^q([0,T];\cL(\R^d;\R^d))$ and $R^F \in D_2^r(\Delta_{[0,T]};\R^d)$, where the remainder $R^F$ is defined implicitly by the relation
  \begin{equation*}
    F_{s,t} = F'_s Z_{s,t} + R^F_{s,t}, \qquad (s,t) \in \Delta_{[0,T]}.
  \end{equation*}
  We refer to $F'$ as the \emph{Gubinelli derivative} of $F$ (with respect to $Z$), and denote the space of such controlled paths by $\crpZ$.
\end{definition}

Given a path $Z \in D^p([0,T];\R^d)$, the space of controlled paths $\crpZ$ becomes a Banach space when equipped with the norm $(F,F') \mapsto |F_0| + \|F,F'\|_{\crpZ,[0,T]}$, where
$$
  \|F,F'\|_{\crpZ,[0,T]} \hspace{1pt} := |F'_0| + \|F'\|_{q,[0,T]} + \|R^F\|_{r,[0,T]}.
$$

With the concepts of rough paths and controlled paths at hand we are ready to introduce rough integration. The following result is a straightforward extension of \cite[Lemma~2.6]{Allan2023}, and its proof follows almost verbatim.

\begin{proposition}\label{prop: general rough integral exists}
  Let $\bX = (X,Z,\X) \in \rp$ be a c\`adl\`ag rough path, and let $(F,F') \in \crpZ$ and $(G,G') \in \crpX$ be controlled paths with respect to $Z$ and $X$, respectively, with remainders $R^F$ and $R^G$. Then, for each $t\in [0,T]$, the limit
  \begin{equation}\label{eq:general rough integral defn}
    \int_0^t F_u \dd G_u := \lim_{|\mathcal{P}| \to 0} \sum_{[u,v] \in \mathcal{P}} F_u G_{u,v} + F'_u G'_u \X_{u,v}
  \end{equation}
  exists along every sequence of partitions $\mathcal{P}$ of the interval $[0,t]$ with mesh size $|\mathcal{P}|$ tending to zero. We call this limit the \emph{rough integral} of $(F,F')$ against $(G,G')$ (relative to the rough path $\bX$), which moreover comes with the estimate
  \begin{align}\label{eq:est int of controlled paths}
    \begin{split}
    &\bigg|\int_s^t F_u \dd G_u - F_s G_{s,t} - F'_s G'_s \X_{s,t}\bigg|\\
    &\quad \leq C\Big(\|F'\|_\infty (\|G'\|_{q,[s,t]}^q + \|Z\|_{p,[s,t]}^p)^{\frac{1}{r}} \|X\|_{p,[s,t]} + \|F\|_{p,[s,t]} \|R^G\|_{r,[s,t]}\\
    &\quad\quad\quad\quad+ \|R^F\|_{r,[s,t]} \|G'\|_\infty \|X\|_{p,[s,t]} + \|F'G'\|_{q,[s,t]} \|\X\|_{\frac{p}{2},[s,t]}\Big),
    \end{split}
  \end{align}
  for all $(s,t) \in \Delta_{[0,T]}$ where the constant $C$ depends only on $p, q$ and $r$.
\end{proposition}

For us, the product of vectors $F_u G_{u,v}$ appearing in \eqref{eq:general rough integral defn} will usually be interpreted as the Euclidean inner product, but in general this product may be interpreted as a matrix of any desired shape and size, consisting of linear combinations of products of the components of the two vectors, and the product $F'_u G'_u \X_{u,v}$ will be a matrix with the same shape.

\begin{remark}\label{remark classical rough integral as special case}
  In the special case when $G = X$ (so that $G'$ is the identity map and $R^G = 0$), the integral defined in Proposition~\ref{prop: general rough integral exists} reduces to the more classical notion of the rough integral of the controlled path $(F,F')$ against the rough path $\bX$, given by
  \begin{equation*}
    \int_0^t F_u \dd \bX_u = \lim_{|\mathcal{P}| \to 0} \sum_{[u,v] \in \mathcal{P}} F_u X_{u,v} + F'_u \X_{u,v}.
  \end{equation*}
\end{remark}

\begin{remark}\label{rmk: rough int is controlled path}
  It follows from the estimate in \eqref{eq:est int of controlled paths}, combined with the relation $G_{s,t} = G'_s X_{s,t} + R^G_{s,t}$, that the rough integral $\int_0^\cdot F_u \dd G_u$ is itself a controlled path with respect to $X$, with Gubinelli derivative $FG'$, so that $(\int_0^\cdot F_u \dd G_u,FG') \in \mathcal{V}^{q,r}_X$.
\end{remark}

Notice that the construction of the rough integral in \eqref{eq:general rough integral defn} is based on so-called compensated Riemann sums $\sum_{[u,v] \in \mathcal{P}} F_u G_{u,v} + F'_u G'_u \X_{u,v}$ instead of classical left-point Riemann sums $\sum_{[u,v] \in \mathcal{P}} F_u G_{u,v}$. While the classical Riemann sums come with a natural interpretation as capital gain processes in the context of mathematical finance, the interpretation of compensated Riemann sums is by no means obvious. However, one advantage of rough integration is that it provides rather powerful stability estimates, for instance as presented in the next proposition.

\begin{proposition}\label{prop general rough int is cts}
  Let $\bX = (X,Z,\X), \tbX = (\tX,\tZ,\tbbX) \in \rp$ be c\`adl\`ag rough paths, and let $(F,F') \in \crpZ$ and $(\tF,\tF') \in \crptZ$ be controlled paths with remainders $R^F$ and $R^{\tF}$ respectively.
  \begin{enumerate}[(i)]
    \item  We have the estimate
    \begin{align*}
      &\bigg\|\int_0^\cdot F_u \dd \bX_u - \int_0^\cdot \tF_u \dd \tbX_u\bigg\|_{p;[0,T]}\\
      &\hspace{2pt}\leq C \Big((|\tF_0| + \|\tF,\tF'\|_{\crptZ,[0,T]})(1 + \|X\|_{p;[0,T]} + \|\tZ\|_{p;[0,T]})\|\bX;\tbX\|_{p;[0,T]}\\
      &\hspace{20pt} + (|F_0 - \tF_0| + |F'_0 - \tF'_0| + \|F' - \tF'\|_{q;[0,T]} + \|R^F - R^{\tF}\|_{r;[0,T]})(1 + \|Z\|_{p;[0,T]})\|\bX\|_{p;[0,T]}\Big),
    \end{align*}
    where the constant $C$ depends only on $p, q$ and $r$.
    \item Let $(G,G') \in \crpX$ and $(\tG,\tG') \in \crptX$ also be controlled paths with remainders $R^G$ and $R^{\tG}$ respectively. Let $M > 0$ be a constant such that
    \begin{align*}
      \|\bX\|_{p,[0,T]}, \|\tbX\|_{p,[0,T]} &\leq M,\\
      |F_0|, \|F,F'\|_{\crpZ,[0,T]}, \|G,G'\|_{\crpX,[0,T]}, |\tF_0|, \|\tF,\tF'\|_{\crptZ,[0,T]}, \|\tG,\tG'\|_{\crptX,[0,T]} &\leq M.
    \end{align*}
    We then have the estimate
    \begin{align*}
      \begin{split}
      &\bigg\|\int_0^\cdot F_u \dd G_u - \int_0^\cdot \tF_u \dd \tG_u\bigg\|_{p,[0,T]}\\
      &\quad\leq C \Big(|F_0 - \tF_0| + |F'_0 - \tF'_0| + \|F' - \tF'\|_{q,[0,T]} + \|R^F - R^{\tF}\|_{r,[0,T]}\\
      &\hspace{43pt} + |G'_0 - \tG'_0| + \|G' - \tG'\|_{q,[0,T]} + \|R^G - R^{\tG}\|_{r,[0,T]} + \|\bX;\tbX\|_{p,[0,T]}\Big)
      \end{split}
    \end{align*}
    where the new constant $C$ depends on $p, q, r$ and $M$.
  \end{enumerate}
\end{proposition}

\begin{proof}
  We present here only the proof of part~(ii), since the proof of part~(i) follows almost verbatim. Here, the multiplicative constant implied by the symbol $\lesssim$ will be allowed to depend on the numbers $p, q$ and $r$ as usual, and additionally on the constant $M$. Following the proof of \cite[Lemma~3.4]{Friz2018}, in our more general setting one deduces the estimates
  \begin{equation}\label{eq:est Y tY}
    \|F - \tF\|_{p,[0,T]} \lesssim |F'_0 - \tF'_0| + \|F' - \tF'\|_{q,[0,T]} + \|R^F - R^{\tF}\|_{r,[0,T]} + \|Z - \tZ\|_{p,[0,T]},
  \end{equation}
  and
  \begin{align}\label{eq:est RintYX RinttYtX}
    \begin{split}
    &\|R^{\int_0^\cdot F_u \dd G_u} - R^{\int_0^\cdot \tF_u \dd \tG_u}\|_{r,[0,T]}\\
    &\quad\lesssim |F_0 - \tF_0| + \|F - \tF\|_{p,[0,T]} + |F'_0 - \tF'_0| + \|F' - \tF'\|_{q,[0,T]} + \|R^F - R^{\tF}\|_{r,[0,T]}\\
    &\hspace{50pt} + |G'_0 - \tG'_0| + \|G' - \tG'\|_{q,[0,T]} + \|R^G - R^{\tG}\|_{r,[0,T]} + \|\bX;\tbX\|_{p,[0,T]}.
    \end{split}
  \end{align} 
  Recalling Remark~\ref{rmk: rough int is controlled path}, we find, using the controlled path structure of the rough integrals, that
  \begin{align}\label{eq:est intYX inttYtX}
    \begin{split}
      \bigg\|\int_0^\cdot F_u \dd G_u - \int_0^\cdot \tF_u \dd \tG_u\bigg\|_{p,[0,T]} \lesssim |&F_0 - \tF_0| + \|F - \tF\|_{p,[0,T]} + |G'_0 - \tG'_0| + \|G' - \tG'\|_{q,[0,T]}\\
      &+ \|X - \tX\|_{p,[0,T]} + \|R^{\int_0^\cdot F_u \dd G_u} - R^{\int_0^\cdot \tF_u \dd \tG_u}\|_{r,[0,T]}.
    \end{split}
  \end{align}
  The result then follows upon substituting the estimates \eqref{eq:est Y tY} and \eqref{eq:est RintYX RinttYtX} into \eqref{eq:est intYX inttYtX}.
\end{proof}

In the spirit of F{\"o}llmer's assumption of quadratic variation along a sequence of partitions~\cite{Follmer1981}, we introduce the following property.

\begin{property}[\textbf{RIE}]
  Let $p\in (2,3)$ and let $\mathcal{P}^n = \{0 = t^n_0 < t^n_1 < \cdots < t^n_{N_n} = T\}$, $n \in \N$, be a sequence of nested partitions of the interval $[0,T]$ such that $|\mathcal{P}^n| \to 0$ as $n \to \infty$. For $S \in D([0,T];\R^d)$, we define $S^n \colon [0,T] \to \R^d$ by
  \begin{equation*}
    S^n_t = S_T \1_{\{T\}}(t) + \sum_{k=0}^{N_n - 1} S_{t^n_k} \1_{[t^n_k,t^n_{k+1})}(t), \qquad t \in [0,T],
  \end{equation*}
  for each $n\in \N$. We assume that:
  \begin{itemize}
    \item the sequence of paths $(S^n)_{n \in \N}$ converges uniformly to $S$ as $n \to \infty$,
    \item the Riemann sums $\int_0^t S^n_u \otimes \d S_u := \sum_{k=0}^{N_n-1} S_{t^n_k} \otimes S_{t^n_k \wedge t,t^n_{k+1} \wedge t}$ converge uniformly as $n \to \infty$ to a limit, which we denote by $\int_0^t S_u \otimes \d S_u$, $t \in [0,T]$,
    \item and that there exists a control function $w$ such that
\begin{equation}\label{eq:RIE sups}
\sup_{(s,t) \in \Delta_{[0,T]}} \frac{|S_{s,t}|^p}{w(s,t)} + \sup_{n \in \N} \sup_{0 \leq k < \ell \leq N_n} \frac{|\int_{t^n_k}^{t^n_\ell} S^n_u \otimes \d S_u - S_{t^n_k}\otimes S_{t^n_k,t^n_\ell}|^{\frac{p}{2}}}{w(t^n_k,t^n_\ell)} \leq 1.
\end{equation}
  \end{itemize}
\end{property}

In \eqref{eq:RIE sups}, and hereafter, we adopt the convention that $\frac{0}{0} := 0$.

\begin{definition}\label{def: S satisfies RIE}
  A path $S\in D([0,T];\R^d)$ is said to satisfy \textup{(RIE)} with respect to $p$ and $(\mathcal{P}^n)_{n\in \N}$, if $p$, $(\mathcal{P}^n)_{n \in \N}$ and $S$ together satisfy Property~\textup{(RIE)}.
\end{definition}

The name ``\textup{RIE}'' is an abbreviation for ``Riemann'', as we assume the convergence of the Riemann sums $\int S^n_u \otimes \d S_u$, instead of the discrete quadratic variations as in \cite{Follmer1981}. Indeed, Property~\textup{(RIE)} is a stronger assumption than the existence of quadratic variation in the sense of F{\"o}llmer, and is even enough to allow us to lift $S$ in a canonical way to a rough path---see Lemma~\ref{lem: SSA is rough path} below---giving us access to the powerful stability results of rough path theory, such as those exhibited in Proposition~\ref{prop general rough int is cts}. Moreover, Property~\textup{(RIE)} can be verified for most typical stochastic processes in mathematical finance, as we will see in Section~\ref{sec:stochastic processes}.

\begin{remark}
We highlight that, rather than simply being a property of a path, Property \textup{(RIE)} is a property of a path together with a given sequence of partitions $(\cP^n)_{n \in \N}$. Indeed, such a path will in general not satisfy \textup{(RIE)} with respect to a different sequence of partitions. However, in practice there is often a natural choice for the sequence of partitions; see Remark~\ref{remark: RIE lift is Ito lift}. For clarity, hereafter, whenever we claim that a path satisfies Property \textup{(RIE)}, we will always make explicit the partition with respect to which the path satisfies \textup{(RIE)}, in the sense of Definition~\ref{def: S satisfies RIE}.
\end{remark}

\begin{remark}
  In Proposition~\ref{prop: uniform convergence} below, it is actually shown that it is sufficient in Property~\textup{(RIE)} to assume that the sequence $(S^n)_{n \in \N}$ converges only pointwise to $S$, since the uniformity of this convergence then immediately follows.
\end{remark}

Next we shall verify that Property~(RIE) ensures the existence of a c{\`a}dl{\`a}g rough path. For this purpose, we consider a suitable approximating sequence of the so-called `area process', which is represented by $\mathbb{X}$ in Definition~\ref{def:rough path}.

\begin{lemma}\label{lem: uniform bound on An}
Suppose $S \in D([0,T];\R^d)$ satisfies Property \textup{(RIE)} with respect to $p$ and $(\mathcal{P}^n)_{n \in \N}$ (as in Definition~\ref{def: S satisfies RIE}). If for each $n \in \N$ we define $A^n \colon \Delta_{[0,T]} \to \R^{d \times d}$ by
  \begin{equation}\label{eq:defn An}
    A^n_{s,t} := \int_s^t S^n_{s,u} \otimes \d S_u = \int_s^t S^n_u \otimes \d S_u - S^n_s \otimes S_{s,t}, \qquad (s,t) \in \Delta_{[0,T]},
  \end{equation}
where $\int_s^t S^n_{s,u} \otimes \d S_u $ is defined as in Property \textup{(RIE)}, then there exists a constant $C$, depending only on $p$, such that
  \begin{equation}\label{eq:An p2var bound}
    \|A^n\|_{\frac{p}{2},[0,T]} \leq Cw(0,T)^{\frac{2}{p}} \qquad \text{for every} \quad n \in \N.
  \end{equation}
\end{lemma}

\begin{proof}
  Let $n \in \N$ and $(s,t) \in \Delta_{[0,T]}$. If there exists a $k$ such that $t^n_k \leq s < t \leq t^n_{k+1}$ then we simply have that $A^n_{s,t} = S_{t^n_k}\otimes S_{s,t} - S_{t^n_k}\otimes S_{s,t} = 0$. Otherwise, let $k_0$ be the smallest $k$ such that $t^n_k \in (s,t)$, and let $k_1$ be the largest such $k$. It is easy to see that the triplet $(S,S^n,A^n)$ satisfies Chen's relation~\eqref{eq:Chens relation}, from which it follows that
  \begin{equation*}
    A^n_{s,t} = A^n_{s,t^n_{k_0}} + A^n_{t^n_{k_0},t^n_{k_1}} + A^n_{t^n_{k_1},t} + S^n_{s,t^n_{k_0}}\otimes S_{t^n_{k_0},t^n_{k_1}} + S^n_{s,t^n_{k_1}}\otimes S_{t^n_{k_1},t}.
  \end{equation*}
  As we have already observed, we have that $A^n_{s,t^n_{k_0}} = A^n_{t^n_{k_1},t} = 0$. By the inequality~\eqref{eq:RIE sups}, we have
\begin{equation*}
|A^n_{t^n_{k_0},t^n_{k_1}}|^{\frac{p}{2}} \leq w(t^n_{k_0},t^n_{k_1}) \leq w(t^n_{k_0-1},t).
\end{equation*}
  We estimate the remaining terms as
  \begin{align*}
    |S^n_{s,t^n_{k_0}}\otimes S_{t^n_{k_0},t^n_{k_1}}|^{\frac{p}{2}} + |S^n_{s,t^n_{k_1}} \otimes S_{t^n_{k_1},t}|^{\frac{p}{2}} &\lesssim |S^n_{s,t^n_{k_0}}|^p + |S_{t^n_{k_0},t^n_{k_1}}|^p + |S^n_{s,t^n_{k_1}}|^p + |S_{t^n_{k_1},t}|^p\\
    &= |S_{t^n_{k_0-1},t^n_{k_0}}|^p + |S_{t^n_{k_0},t^n_{k_1}}|^p + |S_{t^n_{k_0-1},t^n_{k_1}}|^p + |S_{t^n_{k_1},t}|^p\\
    &\leq w(t^n_{k_0-1},t^n_{k_0}) + w(t^n_{k_0},t^n_{k_1}) + w(t^n_{k_0-1},t^n_{k_1}) + w(t^n_{k_1},t)\\
    &\leq 2w(t^n_{k_0-1},t),
  \end{align*}
  so that, putting this all together, we deduce the existence of a constant $\tilde{C} > 0$ such that $|A^n_{s,t}|^{\frac{p}{2}} \leq \tilde{C}w(t^n_{k_0-1},t)$. Taking an arbitrary partition $\mathcal{P}$ of the interval $[0,T]$, it follows that $\sum_{[s,t] \in \mathcal{P}} |A^n_{s,t}|^{\frac{p}{2}} \leq 2\tilde{C}w(0,T)$. We thus conclude that \eqref{eq:An p2var bound} holds with $C = (2\tilde{C})^{\frac{2}{p}}$.
\end{proof}

\begin{lemma}\label{lem: SSA is rough path}
  Suppose that $S\in D([0,T];\R^d)$ satisfies Property~\textup{(RIE)} with respect to $p$ and $(\mathcal{P}^n)_{n \in \N}$. With the natural notation $\int_s^t S_u \otimes \d S_u := \int_0^t S_u \otimes \d S_u - \int_0^s S_u \otimes \d S_u$, we define $A \colon \Delta_{[0,T]} \to \R^{d \times d}$ by 
  \begin{equation*}
    A_{s,t} = \int_s^t S_u \otimes \d S_u - S_s \otimes S_{s,t}, \qquad (s,t) \in \Delta_{[0,T]}.
  \end{equation*}
  Then, the triplet $\bS = (S,S,A)$ is a c{\`a}dl{\`a}g $p$-rough path.
\end{lemma}

\begin{proof}
  It is straightforward to verify Chen's relation~\eqref{eq:Chens relation}, i.e.~that 
  $$
    A_{s,t} = A_{s,u} + A_{u,t} + S_{s,u} \otimes S_{u,t}, \qquad (s,t)\in \Delta_{[0,T]}.
  $$ 
  By Property~\textup{(RIE)}, we know that $\lim_{n \to \infty} A^n_{s,t} = A_{s,t}$, where the convergence is uniform in~$(s,t)$, and thus, being a uniform limit of c{\`a}dl{\`a}g functions, $A$ is itself c{\`a}dl{\`a}g. By the lower semi-continuity of the $\frac{p}{2}$-variation norm, and the result of Lemma~\ref{lem: uniform bound on An}, we have that
  \begin{equation*}
    \|A\|_{\frac{p}{2},[0,T]} \leq \liminf_{n \to \infty} \|A^n\|_{\frac{p}{2},[0,T]} \leq Cw(0,T)^{\frac{2}{p}} < \infty.
  \end{equation*}
  It follows that $(S,S,A)$ is a c{\`a}dl{\`a}g $p$-rough path.
\end{proof}

\subsection{The rough integral as a limit of Riemann sums}\label{subsec:integration under RIE}

While the rough integral in \eqref{eq:general rough integral defn} is a powerful tool to study various differential equations, it lacks the natural interpretation as the capital gain process in the context of mathematical finance. The aim of this subsection is to restore this interpretation by showing that the rough integral can be obtained as the limit of left-point Riemann sums provided that the integrator satisfies Property~\textup{(RIE)}. As preparation we need the following approximation result.

\begin{proposition}\label{prop: uniform convergence}
  Let $\mathcal{P}^n = \{0 = t^n_0 < t^n_1 < \cdots < t^n_{N_n} = T\}$, $n \in \N$, be a sequence of nested partitions with vanishing mesh size, so that $\mathcal{P}^n \subset \mathcal{P}^{n+1}$ for all $n$, and $|\mathcal{P}^n| \to 0$ as $n \to \infty$ (as in the setting of Property~\textup{(RIE)}). Let $F \colon [0,T] \to \R^d$ be a c{\`a}dl{\`a}g path, and define
  \begin{equation}\label{eq:defn Fn}
    F^n_t = F_T \1_{\{T\}}(t) + \sum_{k=0}^{N_n - 1} F_{t^n_k} \1_{[t^n_k,t^n_{k+1})}(t), \qquad t \in [0,T].
  \end{equation}
  Let
  \begin{equation}\label{eq:defn jump times of F}
    J_F := \{t \in (0,T] : F_{t-,t} \neq 0\}
  \end{equation}
  be the set of jump times of $F$. The following are equivalent:
  \begin{enumerate}
    \item[(i)] $J_F \subseteq \cup_{n \in \N} \mathcal{P}^n$,
    \item[(ii)] The sequence $(F^n)_{n \in \N}$ converges pointwise to $F$,
    \item[(iii)] The sequence $(F^n)_{n \in \N}$ converges uniformly to $F$.
  \end{enumerate}
\end{proposition}

\begin{proof}
  We first show that conditions (i) and (ii) are equivalent. To this end, suppose that $J_F \subseteq \cup_{n \geq 1} \mathcal{P}^n$ and let $t \in (0,T]$. If $t \in J_F$, then there exists $m \geq 1$ such that $t \in \mathcal{P}^n$ for all $n \geq m$. In this case we then have that $F^n_t = F_t$ for all $n \geq m$. If $t \notin J_F$, then $F_{t-} = F_t$, and since the mesh size $|\mathcal{P}^n| \to 0$, it follows that $F^n_t \to F_{t-} = F_t$ as $n \to \infty$.

  Now suppose instead that there exists a $t \in J_F$ such that $t \notin \cup_{n \geq 1} \mathcal{P}^n$. We then observe that $F^n_t \to F_{t-} \neq F_t$, so that $F^n_t \nrightarrow F_t$. This establishes the equivalence of (i) and (ii).

  Since (iii) clearly implies (ii), it only remains to show that (ii) implies (iii). By \cite[Theorem~3.3]{Frankova2019}, it is enough to show that the family of paths $\{F^n : n \geq 1\}$ is equiregulated in the sense of \cite[Definition~3.1]{Frankova2019}.

  \textit{Step~1.}
  Let $t \in (0,T]$ and $\epsilon > 0$. Since the left limit $F_{t-}$ exists, there exists $\delta > 0$ with $t - \delta > 0$, such that
  \begin{equation*}
    |F_{s,t-}| < \frac{\epsilon}{2} \qquad \text{for all} \quad s \in (t - \delta,t).
  \end{equation*}
  Let
  $$
    m = \min \{n \geq 1 : \exists k \hspace{7pt} \text{such that} \hspace{7pt} t^n_k \in (t - \delta,t)\}.
  $$
  Since $|\mathcal{P}^n| \to 0$ as $n \to \infty$, we know that $m < \infty$. Moreover, since the sequence of partitions is nested, we immediately have that, for all $n \geq m$, there exists a $k$ such that $t^n_k \in (t - \delta,t)$. We define
  $$
    u = \min \{t^m_k \in \mathcal{P}^m : t^m_k \in (t - \delta,t)\} \equiv \min (\mathcal{P}^m \cap (t - \delta,t)),
  $$
  and let $s \in [u,t)$ and $n \geq 1$.

  If $n < m$, then there does not exist a $k$ such that $t^n_k \in (t - \delta,t)$, which implies that $F^n$ is constant on the interval $(t - \delta,t)$, and hence that $F^n_s = F^n_{t-}$.

  Suppose instead that $n \geq m$. Let $i = \max \{k : t^n_k \leq s\}$ and $j = \max \{k : t^n_k < t\}$. By the definition of $u$, we see that $t^n_i \in [u,t)$ and $t^n_j \in [u,t)$. Then
  \begin{equation*}
    |F^n_s - F^n_{t-}| = |F_{t^n_i} - F_{t^n_j}| \leq |F_{t^n_i,t-}| + |F_{t^n_j,t-}| < \frac{\epsilon}{2} + \frac{\epsilon}{2} = \epsilon.
  \end{equation*}
  Thus, we have that $|F^n_s - F^n_{t-}| < \epsilon$ for all $s \in [u,t)$ and all $n \geq 1$.

  \textit{Step~2.}
  Let $t \in (J_F \cup \{0\}) \setminus \{T\}$ and $\epsilon > 0$. Since $F$ is right-continuous, there exists a $\delta > 0$ with $t + \delta < T$, such that
  \begin{equation*}
    |F_{t,s}| < \epsilon \qquad \text{for all} \quad s \in [t,t + \delta).
  \end{equation*}
  By part~(i), we know that $t \in \cup_{n \geq 1} \mathcal{P}^n$. Let
  $$
    m = \min \{n \geq 1 : \exists k \hspace{7pt} \text{such that} \hspace{7pt} t^n_k = t\}.
  $$
  Since $t \in \cup_{n \geq 1} \mathcal{P}^n$, it is clear that $m < \infty$. We define
  $$
    u = \min \{t^m_k \in \mathcal{P}^m : t^m_k > t\} \equiv \min (\mathcal{P}^m \cap (t,T]).
  $$
  We then let $v \in (t,u \wedge (t + \delta))$, $s \in (t,v]$, and $n \geq 1$.

  If $n < m$, then, since $v < u$, there does not exist a $k$ such that $t^n_k \in [t,v]$. Hence, $F^n$ is constant on the interval $[t,v]$, so that in particular $F^n_s = F^n_t$.

  Suppose instead that $n \geq m$. By the definition of $m$, there exists a~$j$ such that $t^n_j = t$. Let $i = \max \{k : t^n_k \leq s\}$. In particular, we then have that $t = t^n_j \leq t^n_i \leq s \leq v < t + \delta$, and hence that
  \begin{equation*}
    |F^n_s - F^n_t| = |F_{t^n_i} - F_{t^n_j}| = |F_{t,t^n_i}| < \epsilon.
  \end{equation*}
  Thus, we have that $|F^n_s - F^n_t| < \epsilon$ for all $s \in (t,v]$ and all $n \geq 1$.

  \textit{Step~3.}
  Let $t \in (0,T) \setminus J_F$ and $\epsilon > 0$. Since~$F$ is continuous at time $t$, there exists a $\delta > 0$ with $0 < t - \delta$ and $t + \delta < T$, such that
  \begin{equation*}
    |F_{s,t}| < \frac{\epsilon}{2} \qquad \text{for all} \quad s \in (t - \delta,t + \delta).
  \end{equation*}
  Let
  $$
    m = \min \{n \geq 1 : \exists k \hspace{7pt} \text{such that} \hspace{7pt} t^n_k \in (t - \delta,t]\}.
  $$
  Since $|\mathcal{P}^n| \to 0$ as $n \to \infty$, we know that $m < \infty$. We define
  $$
    u = \min \{t^m_k \in \mathcal{P}^m : t^m_k > t\} \equiv \min (\mathcal{P}^m \cap (t,T]).
  $$
  We then let $v \in (t,u \wedge (t + \delta))$, $s \in (t,v]$ and $n \geq 1$.

  If $n < m$, then, since $v < u$, there does not exist a~$k$ such that $t^n_k \in (t,v]$. Hence, $F^n$ is constant on the interval $[t,v]$, so that in particular $F^n_s = F^n_t$.

  Suppose instead that $n \geq m$. Let $i = \max \{k : t^n_k \leq s\}$ and $j = \max \{k : t^n_k \leq t\}$. Since, by the definition of $m$, there exists at least one $k$ such that $t^n_k \in (t - \delta,t]$, and since $t < s \leq v < t + \delta$, it follows that $t^n_i \in (t - \delta,t + \delta)$ and $t^n_j \in (t - \delta,t]$. Then
  \begin{equation*}
    |F^n_s - F^n_t| = |F_{t^n_i} - F_{t^n_j}| \leq |F_{t^n_i,t}| + |F_{t^n_j,t}| < \frac{\epsilon}{2} + \frac{\epsilon}{2} = \epsilon.
  \end{equation*}
  Thus, we have that $|F^n_s - F^n_t| < \epsilon$ for all $s \in (t,v]$ and all $n \geq 1$. It follows that the family of paths $\{F^n : n \geq 1\}$ is indeed equiregulated.
\end{proof}

The next theorem is the main result of this section, stating that the rough integral can be approximated by left-point Riemann sums along a suitable sequence of partitions, in the spirit of F{\"o}llmer's pathwise integration.

\begin{theorem}\label{thm: rough int as limit Riemann sums}
  Let $q \geq p$ such that $\frac{2}{p} + \frac{1}{q} > 1$, and let $r > 1$ such that $\frac{1}{r} = \frac{1}{p} + \frac{1}{q}$. Suppose that $S \in D([0,T];\R^d)$ satisfies Property~\textup{(RIE)} with respect to $p$ and $(\mathcal{P}^n)_{n \in \N}$. Let $(F,F') \in \mathcal{V}^{q,r}_S$ and $(G,G') \in \mathcal{V}^{q,r}_S$ be controlled paths with respect to $S$, and assume that $J_F \subseteq \cup_{n \in \N} \mathcal{P}^n$, where $J_F$ is the set of jump times of $F$, as in \eqref{eq:defn jump times of F}. Then the rough integral of $(F,F')$ against $(G,G')$ relative to the rough path $\bS = (S,S,A)$, as defined in \eqref{eq:general rough integral defn}, is given by
  \begin{equation*}
    \int_0^t F_u \dd G_u = \lim_{n \to \infty} \sum_{k=0}^{N_n-1} F_{t^n_k} G_{t^n_k \wedge t,t^n_{k+1} \wedge t},
  \end{equation*}
  where the convergence is uniform in $t \in [0,T]$.
\end{theorem}

\begin{proof}
  We recall from Lemma~\ref{lem: SSA is rough path} that $\bS = (S,S,A)$ is a $p$-rough path, so that, by Proposition~\ref{prop: general rough integral exists}, the rough integral of $(F,F')$ against $(G,G')$ (relative to $\bS$) exists. It is also clear that $\bS^n := (S,S^n,A^n)$ is a $p$-rough path, where $A^n$ was defined in \eqref{eq:defn An}. Moreover, by Property~\textup{(RIE)}, we immediately have that $S^n$ and $A^n$ converge uniformly to $S$ and $A$ respectively as $n \to \infty$.

  For each $n \geq 1$, we let $F^n$ be the path defined in~\eqref{eq:defn Fn}. We consider the pair $(F^n,F')$ as a controlled path with respect to $S^n$, defining the remainder term $R^n$ by the usual relation:
  \begin{equation*}
    F^n_{s,t} = F'_s S^n_{s,t} + R^n_{s,t}, \qquad (s,t) \in \Delta_{[0,T]}.
  \end{equation*}
  Since $S^n$ converges uniformly to $S$ and, by Proposition~\ref{prop: uniform convergence}, $F^n$ converges uniformly to $F$, it follows that $R^n$ also converges uniformly to the remainder term $R$ corresponding to the $S$-controlled path $(F,F')$.

  We observe that $\|S^n\|_{p,[0,T]} \leq \|S\|_{p,[0,T]}$ and $\|F^n\|_{p,[0,T]} \leq \|F\|_{p,[0,T]}$, and we have from Lemma~\ref{lem: uniform bound on An} that $\|A^n\|_{\frac{p}{2},[0,T]} \leq Cw(0,T)^{\frac{2}{p}}$ for every $n \geq 1$. It remains to show that $R^n$ is bounded in $r$-variation, uniformly in $n$.

  Let $n \geq 1$ and $(s,t) \in \Delta_{[0,T]}$. If there exists a~$k$ such that $t^n_k \leq s < t < t^n_{k+1}$, then
  $$
    R^n_{s,t} = F^n_{s,t} - F'_s S^n_{s,t} = F_{t^n_k,t^n_k} - F'_s S_{t^n_k,t^n_k} = 0.
  $$
  If there exists a $k$ such that $t^n_k \leq s < t = t^n_{k+1}$, then
  \begin{align*}
    |R^n_{s,t}|^r &= |F^n_{s,t} - F'_s S^n_{s,t}|^r = |F_{t^n_k,t^n_{k+1}} - F'_s S_{t^n_k,t^n_{k+1}}|^r\\
    &\lesssim |F_{t^n_k,t^n_{k+1}} - F'_{t^n_k} S_{t^n_k,t^n_{k+1}}|^r + |F'_{t^n_k,s} S_{t^n_k,t^n_{k+1}}|^r\\
    &\lesssim |R_{t^n_k,t^n_{k+1}}|^r + |F'_{t^n_k,s}|^q + |S_{t^n_k,t^n_{k+1}}|^p,
  \end{align*}
  where in the last line we used Young's inequality, recalling that $\frac{1}{r} = \frac{1}{p} + \frac{1}{q}$.

  Otherwise, let $k_0$ be the smallest~$k$ such that $t^n_k \in [s,t]$, and let $k_1$ be the largest such~$k$. After a short calculation, we find that
  \begin{equation*}
    R^n_{s,t} = R^n_{s,t^n_{k_0}} + R^n_{t^n_{k_0},t^n_{k_1}} + R^n_{t^n_{k_1},t} + F'_{s,t^n_{k_0}} S_{t^n_{k_0},t^n_{k_1}} + F'_{s,t^n_{k_1}} S^n_{t^n_{k_1},t}.
  \end{equation*}
  We observe that $S^n_{t^n_{k_1},t} = 0$ and $R^n_{t^n_{k_0},t^n_{k_1}} = R_{t^n_{k_0},t^n_{k_1}}$. We can deal with the terms $R^n_{s,t^n_{k_0}}$ and $R^n_{t^n_{k_1},t}$ using the above, and we bound $|F'_{s,t^n_{k_0}} S_{t^n_{k_0},t^n_{k_1}}|^r \lesssim |F'_{s,t^n_{k_0}}|^q + |S_{t^n_{k_0},t^n_{k_1}}|^p$. Putting this all together, we have that
  \begin{equation*}
    |R^n_{s,t}|^r \leq C \Big(|R_{t^n_{k_0-1},t^n_{k_0}}|^r + |F'_{t^n_{k_0-1},s}|^q + |S_{t^n_{k_0-1},t^n_{k_0}}|^p + |R_{t^n_{k_0},t^n_{k_1}}|^r + |F'_{s,t^n_{k_0}}|^q + |S_{t^n_{k_0},t^n_{k_1}}|^p\Big),
  \end{equation*}
  where the constant $C$ depends only on $p, q$ and $r$. Taking an arbitrary partition $\mathcal{P}$ of the interval $[0,T]$, we deduce that $\sum_{[s,t] \in \mathcal{P}} |R^n_{s,t}|^r \leq 2C(\|R\|_{r,[0,T]}^r + \|F'\|_{q,[0,T]}^q + \|S\|_{p,[0,T]}^p)$. Thus, $\|R^n\|_{r,[0,T]}$ is bounded uniformly in $n \geq 1$.

  Let $p' > p$, $q' > q$ and $r' > r$, such that $p' \in (2,3)$, $q' \geq p'$, $\frac{2}{p'} + \frac{1}{q'} > 1$, and $\frac{1}{r'} = \frac{1}{p'} + \frac{1}{q'}$. Since the sequence $(S^n)_{n \geq 1}$ has uniformly bounded $p$-variation, and $S^n$ converges uniformly to $S$ as $n \to \infty$, it follows by interpolation that $S^n$ converges to $S$ with respect to the $p'$-variation norm, i.e.~$\|S^n - S\|_{p',[0,T]} \to 0$ as $n \to \infty$. It follows similarly that $\|A^n - A\|_{\frac{p'}{2},[0,T]} \to 0$ and $\|R^n - R\|_{r',[0,T]} \to 0$, and hence also that $\|\bS^n;\bS\|_{p',[0,T]} \to 0$ as $n \to \infty$. It thus follows from part~(ii) of Proposition~\ref{prop general rough int is cts} that
  \begin{equation}\label{eq:rough integrals converge}
    \int_0^t F^n_u \dd G_u \longrightarrow \int_0^t F_u \dd G_u \qquad \text{as} \quad n \longrightarrow \infty,
  \end{equation}
  where the convergence is uniform in $t \in [0,T]$. Note that in \eqref{eq:rough integrals converge} the integral $\int_0^t F^n_u \dd G_u$ is defined relative to the rough path $\bS^n = (S,S^n,A^n)$, whilst the limiting rough integral $\int_0^t F_u \dd G_u$ is defined relative to $\bS = (S,S,A)$.

  We recall from Proposition~\ref{prop: general rough integral exists} that the integral of $(F^n,F')$ against $(G,G')$ relative to $\bS^n = (S,S^n,A^n)$ is given by the limit
  \begin{equation*}
     \int_0^t F^n_u \dd G_u = \lim_{|\mathcal{P}| \to 0} \sum_{[u,v] \in \mathcal{P}} F^n_u G_{u,v} + F'_u G'_u A^n_{u,v},
  \end{equation*}
  where the limit is taken over any sequence of partitions of the interval $[0,t]$ with vanishing mesh size. Take any refinement $\tilde{\mathcal{P}}$ of the partition $(\mathcal{P}^n \cup \{t\}) \cap [0,t]$ (where as usual $\mathcal{P}^n$ is the partition given in Property~\textup{(RIE)}), and let $[u,v] \in \tilde{\mathcal{P}}$. By the choice of the partition $\tilde{\mathcal{P}}$, there exists a $k$ such that $t^n_k \leq u < v \leq t^n_{k+1}$, which, recalling \eqref{eq:defn An}, implies that $A^n_{u,v} = 0$. Since the mesh size of $\tilde{\mathcal{P}}$ may be arbitrarily small, it follows that
  $$
    \lim_{|\tilde{\mathcal{P}}| \to 0} \sum_{[u,v] \in \tilde{\mathcal{P}}} F'_u G'_u A^n_{u,v} = 0.
  $$
  To conclude, we then simply recall~\eqref{eq:rough integrals converge}, and note that
  \begin{equation*}
     \int_0^t F^n_u \dd G_u = \lim_{|\tilde{\mathcal{P}}| \to 0} \sum_{[u,v] \in \tilde{\mathcal{P}}} F^n_u G_{u,v} = \sum_{k=0}^{N_n-1} F_{t^n_k} G_{t^n_k \wedge t,t^n_{k+1} \wedge t}.
  \end{equation*}
\end{proof}

We can actually generalize the result of Theorem~\ref{thm: rough int as limit Riemann sums} to a slightly larger class of integrands.

\begin{corollary}
  Recall the assumptions of Theorem~\ref{thm: rough int as limit Riemann sums}, and let $\gamma \in D^r([0,T];\R^d)$. Then, the rough integral of the controlled path $H = F + \gamma$, given by $(H,H') := (F + \gamma,F') \in \mathcal{V}^{q,r}_S$, against $(G,G')$ is given by
  \begin{equation*}
    \int_0^t H_u \dd G_u = \lim_{n \to \infty} \sum_{k=0}^{N_n-1} H_{t^n_k} G_{t^n_k \wedge t,t^n_{k+1} \wedge t}
  \end{equation*}
  for every $t \in [0,T]$.
\end{corollary}

The point here is that the path $\gamma$ may have jump times which do not belong to the set $\cup_{n \in \N} \mathcal{P}^n$.

\begin{proof}
  Since $\gamma$ has finite $r$-variation, we immediately have that $\gamma$ is a controlled path with Gubinelli derivative simply given by $\gamma' = 0$. By linearity, it is then clear that $(H,H') = (F,F') + (\gamma,0)$ is indeed a controlled path with respect to $S$. Since $\gamma' = 0$, we have from Proposition~\ref{prop: general rough integral exists} that
  $$
    \int_0^t \gamma_u \dd G_u = \lim_{|\mathcal{P}| \to 0} \sum_{[u,v] \in \mathcal{P}} \gamma_u G_{u,v} = \lim_{n \to \infty} \sum_{k=0}^{N_n-1} \gamma_{t^n_k} G_{t^n_k \wedge t,t^n_{k+1} \wedge t}.
  $$
  By linearity, we have that $\int_0^t H_u \dd G_u = \int_0^t F_u \dd G_u + \int_0^t \gamma_u \dd G_u$, and the result then follows from Theorem~\ref{thm: rough int as limit Riemann sums}.
\end{proof}

\subsection{Link to F{\"o}llmer integration}\label{subsec:Follmer integration}

In his seminal paper~\cite{Follmer1981}, F{\"o}llmer introduced a notion of pathwise integration based on the concept of quadratic variation, and derived a corresponding pathwise It{\^o} formula, which have proved to be useful tools in robust approaches to mathematical finance.

In the following we will write $\mathcal{B}[0,T]$ for the Borel $\sigma$-algebra on $[0,T]$.

\begin{definition}
  Let $S \in D([0,T];\R)$ and let $\mathcal{P}^n = \{0 = t^n_0 < t^n_1 < \cdots < t^n_{N_n} = T\}$, $n \geq 1$, be a sequence of partitions with vanishing mesh size. We say that $S$ has \emph{quadratic variation along $(\mathcal{P}^n)_{n \in\N}$ in the sense of F{\"o}llmer} if the sequence of measures $(\mu_n)_{n \in\N}$ on $([0,T],\mathcal{B}[0,T])$ defined by
  \begin{equation*}
    \mu_n := \sum_{k = 0}^{N_n - 1} |S_{t^n_k,t^n_{k+1}}|^2 \delta_{t^n_k},
  \end{equation*}
  converges weakly to a measure $\mu$, such that the map $t \mapsto [S]^c_t := \mu([0,t]) - \sum_{0 < s \leq t} |S_{s-,s}|^2$ is continuous and increasing. In this case we call the function $[S]$, given by $[S]_t = \mu([0,t])$, the \emph{quadratic variation} of $S$ along $(\mathcal{P}^n)_{n \in\N}$.

  We say that a path $S \in D([0,T];\R^d)$ has \textup{quadratic variation along $(\mathcal{P}^n)_{n \in\N}$ in the sense of F{\"o}llmer} if the condition above holds for $S^i$ and $S^i + S^j$ for every $(i,j)$, and in this case we write
  \begin{equation}\label{eq:defn Follmer bracket}
    [S^i,S^j] := \frac{1}{2} ([S^i + S^j] - [S^i] - [S^j]).
  \end{equation}
\end{definition}

Assuming that a path $S\in D([0,T];\R^d)$ has quadratic variation along $(\mathcal{P}^n)_{n \in\N}$ and $f\in C^2 (\R^d;\R)$, F{\"o}llmer showed that the limit
\begin{equation*}
  \int_0^T \D f(S_u) \dd S_u 
  := \lim_{n\to \infty }\sum_{[s,t]\in \mathcal{P}^n} \D f(S_s) S_{s,t}
\end{equation*}
exists and the resulting integral $\int_0^T \D f(S_u) \dd S_u$ satisfies a pathwise It{\^o} formula, see~\cite[TH{\'E}OR{\`E}ME]{Follmer1981}. Let us remark that the F{\"o}llmer integral~$\int_0^T \D f(S_u) \dd S_u$ is only well-defined for gradients~$\D f$ and not for general functions, as its existence is given by the corresponding pathwise It{\^o} formula. This result can also be explained via the language rough path theory; see Friz and Hairer \cite[Chapter~5.3]{FrizHairer2020}.

\smallskip

In the following we relate Property~\textup{(RIE)} to the existence of quadratic variation in the sense of F{\"o}llmer. To this end, for each $i = 1,\ldots,d$, we introduce 
$$
  S^{n,i}_t = S^i_T \1_{\{T\}}(t) + \sum_{k=0}^{N_n - 1} S^i_{t^n_k} \1_{[t^n_k,t^n_{k+1})}(t)
$$
and the discrete quadratic variation $\langle S^i, S^j \rangle^n$ by
\begin{equation*}
  \langle S^i, S^j \rangle^n_t = \sum_{k=0}^{N_n - 1} S^i_{t^n_k \wedge t,t^n_{k+1} \wedge t} S^j_{t^n_k \wedge t,t^n_{k+1} \wedge t}, \qquad t \in [0,T].
\end{equation*}

\begin{proposition}\label{prop RIE implies Follmer quad var}
  Let $S \in D([0,T];\R^d)$ and let $\mathcal{P}^n = \{0 = t^n_0 < t^n_1 < \cdots < t^n_{N_n} = T\}$, $n \in\N$, be a sequence of nested partitions with vanishing mesh size. The following conditions are equivalent:
  \begin{enumerate}
	\item[(i)] For every pair $(i,j)$, the Riemann sums $\int_0^{t}S^{n,i}_u \dd S^j_u + \int_0^t S^{n,j}_u \dd S^i_u$ converge uniformly to a limit, which we denote by $\int_0^t S^i_u \dd S^j_u + \int_0^t S^j_u \dd S^i_u$.
    \item[(ii)] For every pair $(i,j)$, the discrete quadratic variation $\langle S^i, S^j \rangle^n$ converges uniformly to a c{\`a}dl{\`a}g path, which we denote by $\langle S^i, S^j \rangle$.
    \item[(iii)] The path $S$ has quadratic variation along $(\mathcal{P}^n)_{n \in\N}$ in the sense of F{\"o}llmer.
  \end{enumerate}
  Moreover, if these conditions hold then the path $\langle S^i, S^j \rangle$ has finite total variation, and, for every $(i,j)$, we have that $[S^i,S^j] = \langle S^i,S^j \rangle$ and the equality
  \begin{equation}\label{eq:int by parts formula}
    S^i_t S^j_t = S^i_0 S^j_0 + \int_0^t S^i_u \dd S^j_u + \int_0^t S^j_u \dd S^i_u + \langle S^i, S^j \rangle_t
  \end{equation}
  holds for every $t \in [0,T]$.
\end{proposition}

\begin{proof}
  We have
  \begin{align*}
    &S^i_t S^j_t - S^i_0 S^j_0 = \sum_{k=0}^{N_n - 1} (S^i_{t^n_{k+1} \wedge t} S^j_{t^n_{k+1} \wedge t} - S^i_{t^n_k \wedge t} S^j_{t^n_k \wedge t})\\
    &= \sum_{k=0}^{N_n - 1} (S^i_{t^n_k \wedge t} S^j_{t^n_k \wedge t,t^n_{k+1} \wedge t} + S^j_{t^n_k \wedge t} S^i_{t^n_k \wedge t,t^n_{k+1} \wedge t}) + \sum_{k=0}^{N_n - 1} S^i_{t^n_k \wedge t,t^n_{k+1} \wedge t} S^j_{t^n_k \wedge t,t^n_{k+1} \wedge t}\\
    &= \int_0^t S^{n,i}_u \dd S^j_u + \int_0^t S^{n,j}_u \dd S^i_{u} + \langle S^i, S^j \rangle^n_t,
  \end{align*}
  from which it follows that conditions (i) and (ii) are equivalent, and that \eqref{eq:int by parts formula} then also holds. In this case, we also have that
  \begin{equation*}
    \langle S^i, S^j \rangle_t = \frac{1}{4} \Big(\langle S^i + S^j, S^i + S^j \rangle_t - \langle S^i - S^j, S^i - S^j \rangle_t\Big),
  \end{equation*}
  so that, as the difference of two non-decreasing functions, $\langle S^i, S^j \rangle$ has finite total variation.
  
  For one-dimensional paths~$S$, the equivalence of conditions (ii) and (iii) follows from \cite[Propositions~3 and~4]{Vovk2015}. The extension of this to $d$-dimensional paths $S$ and the equality $[S^i,S^j] = \langle S^i,S^j \rangle$ then follow from the polarization identity
  \begin{equation*}
    \langle S^i,S^j \rangle^n_t = \frac{1}{2} \Big(\langle S^i + S^j,S^i + S^j \rangle^n_t - \langle S^i,S^i \rangle^n_t - \langle S^j,S^j \rangle^n_t\Big)
  \end{equation*}
  and the definition of $[S^i,S^j]$ in \eqref{eq:defn Follmer bracket}.
\end{proof}

\begin{remark}\label{remark; Ito formula}
  As an immediate consequence of Proposition~\ref{prop RIE implies Follmer quad var}, we have that if a path $S$ satisfies~\textup{(RIE)} along $(\mathcal{P}^n)_{n \in \N}$, then it has quadratic variation along $(\mathcal{P}^n)_{n \in \N}$ in the sense of F{\"o}llmer, thus allowing one to apply all the known results regarding F{\"o}llmer integration.

In particular, if a vector field $f \colon \R^d \to \R$ is of class $C^3$, then, by Theorem~\ref{thm: rough int as limit Riemann sums}, the F\"ollmer integral $\int_0^\cdot \D f(S_u) \dd S_u$ coincides with the rough integral $\int_0^\cdot \D f(S_u) \dd \mathbf{S}_u$. We thus obtain the rough It{\^o} formula:
\begin{align*}
f(S_t)- f(S_0) = &\int_0^t \D f(S_u) \dd \mathbf{S}_u + \frac{1}{2} \int_0^t \D^2 f(S_u) \dd [S]_u\\
& + \sum_{0 < u \leq t} \Big(f(S_u) - f(S_{u-}) - \D f(S_{u-}) \Delta S_u - \frac{1}{2} \D^2 f(S_{u-}) (\Delta S_u \otimes \Delta S_u)\Big),
\end{align*}
which holds for every $t \in [0,T]$, where $[S] = ([S^i,S^j])_{1 \leq i, j \leq d}$ denotes the quadratic variation matrix, and $\Delta S_u := \lim_{s \to u,s < u} S_{s,u}$. We note that the formula above is precisely the It\^o formula for rough paths derived in Friz and Zhang \cite{Friz2018}.
\end{remark}

\section{Functionally generated trading strategies and their generalizations}\label{sec:trading strategies}

Given Property~\textup{(RIE)}, we can introduce a model-free framework for  continuous-time financial markets with a possibly infinite time horizon. In this section we shall verify that most relevant trading strategies from a practical perspective, such as delta-hedging strategies and functionally generated strategies, are admissible integrands for price paths satisfying Property~\textup{(RIE)}. Furthermore, the underlying rough integration allows us to deduce stability estimates for admissible strategies.

\subsection{Price paths and admissible strategies}

For a path $S \colon [0,\infty) \to \R^{d}$, we denote by $S|_{[0,T]}$ the restriction of $S$ to the interval $[0,T]$.

\begin{definition}\label{defn price path}
  For a fixed $p \in (2,3)$, we say that a path $S \in D([0,\infty);\R^d)$ is a \emph{price path}, if there exists a nested sequence of locally finite partitions $(\cP^n)_{n \in \N}$ of the interval $[0,\infty)$, with vanishing mesh size on compacts, such that, for all $T > 0$, the restriction $S|_{[0,T]}$ satisfies~\textup{(RIE)} with respect to $p$ and $(\cP^n([0,T]))_{n \in \N}$.

  We denote the family of all such price paths by $\Omega_p$.
\end{definition}

Note that the sequence of partitions $(\cP^n)_{n \in \N}$ may depend on the choice of price path $S \in \Omega_p$, consistent with the stochastic framework where this sequence will naturally be defined in terms of (probabilistic) stopping times.

Having fixed the model-free structure of the underlying price paths, we can introduce the class of admissible strategies and the corresponding capital process.

\begin{definition}\label{def: admissible strategy}
  Let $p \in (2,3)$ and let $S \in \Omega_p$ be a price path. We say that a path $\phi \colon [0,\infty) \to \R^d$ is an \emph{admissible strategy (with respect to $S$)}, if
  \begin{itemize}
    \item there exist $q \geq p$ and $r > 1$ with $2/p + 1/q > 1$ and $1/r = 1/p + 1/q$, such that for every $T > 0$, there exists a path $\phi' \colon [0,T] \to \cL(\R^d;\R^d)$ such that the pair $(\phi,\phi') \in \cV^{q,r}_S$ is a controlled path with respect to $S$ in the sense of Definition~\ref{def: controlled path},
    \item and $J_\phi \subseteq \cup_{n \in \N} \cP^n$, where $J_\phi$ is the set of jump times of $\phi$ in $(0,\infty)$, and $(\cP^n)_{n \in \N}$ is the sequence of partitions associated with the price path $S \in \Omega_p$.
  \end{itemize}
  We denote the space of all admissible strategies (with respect to $S$) by $\mathcal{A}_S$.

  We define the \emph{capital process} associated with $\phi$ and $S$ as the path $V^\phi(S) \colon [0,\infty) \to \R$ given by
  \begin{equation}\label{eq:defn capital process}
    V^\phi_t(S) := \lim_{n \to \infty} \sum_{k=0}^{N_n-1} \sum_{i=1}^d \phi^i_{t^n_k} (S^i_{t^n_{k+1} \wedge t} - S^i_{t^n_k \wedge t}), \qquad t \in [0,\infty),
  \end{equation}
  where $\cP^n = \{0 = t^n_0 < t^n_1 < \cdots < t^n_{N_n} = T\}$ is the sequence of partitions specified in Property~\textup{(RIE)}.
\end{definition}

\begin{remark}
In a semimartingale setting, as often used in classical mathematical finance, one usually considers left-continuous trading strategies $\phi$. In the present setting this assumption is not necessary, as the corresponding capital process $V^{\phi}(S)$, which, as we will see below, may be expressed as a rough integral, does not change when replacing $\phi$ by its left-continuous modification; see e.g.~\cite[Theorem~31]{Friz2017}. The reason for this is essentially the left-point Riemann sum construction of the integral. Indeed, suppose that $S$ has a jump at a time $t > 0$. The contribution to the capital process $V^\phi(S)$ at time $t$ is then given by $\lim_{s \to t,\hspace{1pt} s < t} \phi_s S_{s,t}$, which is invariant to the choice of $\phi$ or its left-continuous modification. Furthermore, we will see that $V^{\phi}(S)$ coincides with the classical stochastic It\^o integral, whenever both the rough and stochastic integrals are defined; see Section~\ref{subsect: rough integral vs stochastic integral} below.

The condition $J_\phi \subseteq \cup_{n \in \N} \cP^n$ means that one is allowed to use trading strategies whose jump points are included in the underlying sequence of partitions $(\mathcal{P}^n)_{n \in \N}$. On the one hand, many frequently used trading strategies, such as delta-hedging, satisfy this condition, and further examples will be discussed later in this section. On the other hand, the sequence of partitions $(\mathcal{P}^n)_{n \in \N}$ can be fixed a priori to allow for a desired class of trading strategies, e.g.~buy and hold strategies along the sequence of dyadic partitions.
\end{remark}

\begin{proposition}\label{prop: capital process}
  Let $p \in (2,3)$, let $S \in \Omega_p$ be a price path, and let $\phi \in \mathcal{A}_S$ be an admissible strategy (in the sense of Definition~\ref{def: admissible strategy}). Then, the capital process $V^\phi(S)$ as defined in \eqref{eq:defn capital process} exists as a locally uniform limit, and is actually given by
  \begin{equation}\label{eq:rough capital process}
    V^\phi_t(S) =  \int_0^t \phi_s \dd \bS_s, \qquad t \in [0,\infty),
  \end{equation}
  that is, the rough integral of the controlled path $(\phi,\phi') \in \cV^{q,r}_S$ against the rough path $\bS$ defined in Lemma~\ref{lem: SSA is rough path}.

  Moreover, given another price path $\tS \in \Omega_p$ and an admissible strategy $\tphi \in \mathcal{A}_{\tilde{S}}$ with respect to $\tS$, we have, for every $T > 0$, that
  \begin{align*}
    &|V^\phi_T(S) - V^{\tphi}_T(\tS)|\\
    &\leq C \Big((|\tphi_0| + \|\tphi,\tphi'\|_{\cV^{q,r}_{\tilde{S}}}) (1 + \|S\|_{p,[0,T]} + \|\tS\|_{p,[0,T]}) \|\bS;\tbS\|_{p,[0,T]}\\
    &\quad + (|\phi_0 - \tphi_0| + |\phi'_0 - \tphi'_0| + \|\phi' - \tphi'\|_{q,[0,T]} + \|R^\phi - R^{\tphi}\|_{r,[0,T]}) (1 + \|S\|_{p,[0,T]}) \|\bS\|_{p,[0,T]}\Big),
  \end{align*}
  where the constant $C$ depends on $p, q$ and $r$.
\end{proposition}

\begin{remark}
  Recall from Remark~\ref{remark classical rough integral as special case} that the rough integral in \eqref{eq:rough capital process} is defined by the limit $\int_0^t \phi_s \dd \bS_s = \lim_{|\pi| \to 0} \sum_{[u,v] \in \pi} \phi_u S_{u,v} + \phi'_u A_{u,v}$, where the limit is taken over any sequence of partitions of the interval $[0,t]$ with vanishing mesh size. Here, $\phi_u$ and $S_{u,v}$ both take values in $\R^d$, and we interpret their multiplication as the Euclidean inner product. The derivative $\phi'_u$ takes values in $\cL(\R^d;\R^d)$, which we can also identify with $\cL(\R^{d \times d};\R)$. Since $A_{u,v} \in \R^{d \times d}$, the product $\phi'_u A_{u,v}$ also takes values in $\R$.
\end{remark}

\begin{proof}[Proof of Proposition~\ref{prop: capital process}]
  Let $T > 0$, and let $(\cP^n)_{n \in \N}$ be a sequence of nested partitions such that $p$, $(\cP^n)_{n \in \N}$ and $S$ satisfy Property~\textup{(RIE)} on the interval $[0,T]$. Recall from Property~\textup{(RIE)} the existence of the limit $\int_0^t S_u \otimes \d S_u$ for every $t \in [0,T]$. By Lemma~\ref{lem: SSA is rough path}, defining the function $A \colon \Delta_{[0,T]} \to \R^{d \times d}$ by
  $$
    A_{s,t} := \int_s^t S_u \otimes \d S_u - S_s \otimes S_{s,t},
  $$
  we have that the triplet $\bS = (S,S,A)$ is a c{\`a}dl{\`a}g rough path (in the sense of Definition~\ref{def:rough path}). Hence, the rough integral in \eqref{eq:rough capital process} is well-defined by Proposition~\ref{prop: general rough integral exists} (see also Remark~\ref{remark classical rough integral as special case}), and satisfies \eqref{eq:defn capital process} as a locally uniform limit by Theorem~\ref{thm: rough int as limit Riemann sums}.

  For the stability estimate we simply note that
  \begin{equation*}
    |V^\phi_T(S) - V^{\tphi}_T(\tS)| = \bigg|\int_0^T \phi_s \dd \bS_s - \int_0^T \tphi_s \dd \tbS_s\bigg| \leq \bigg\|\int_0^\cdot \phi_s \dd \bS_s - \int_0^\cdot \tphi_s \dd \tbS_s\bigg\|_{p,[0,T]}
  \end{equation*}
  and apply part~(i) of Proposition~\ref{prop general rough int is cts}.
\end{proof}

In the following we show that the most relevant trading strategies from a practical viewpoint belong to the class of admissible strategies in the sense of Definition~\ref{def: admissible strategy}.

\subsection{Functionally generated trading strategies}\label{subsec:functionally generated trading strategies}

Having fixed the set $\Omega_p$ of underlying price paths, we start by introducing functionally generated portfolios. For this purpose, for some $d_A \in \N$, we fix a c{\`a}dl{\`a}g path $A \colon [0,\infty) \to \R^{d_A}$ of locally bounded variation and assume that the jump times of $A$ belong to the union of the partitions $(\cP^n)_{n \in \N}$ appearing in Property~\textup{(RIE)}; that is, we assume that $J_A \subseteq \cup_{n \in \N} \cP^n$, where $J_A := \{t \in (0,\infty) : A_{t-,t} \neq 0\}$. The path~$A$ is supposed to include additional information pertaining to the market which a trader would like to include in their trading decisions. For instance, the components of the path $A = (A^1,\dots,A^{d_A})$ could include time $t \mapsto t$, the running maximum $t \mapsto \max_{u \in [0,t]} S^i_u$, or the integral $t \mapsto \int_0^t S^i_u\dd u$ for some (or all) $i=1,\dots,d$. A more detailed discussion on practical choices of the path $A$ can be found in Schied, Speiser and Voloshchenko \cite{Schied2018}.

\smallskip

For $\ell = d + d_A$, we denote by $C^{2}_b(\R^{\ell};\R^d)$ the space of twice continuously differentiable (in the Fr{\'e}chet sense) functions $f \colon \R^\ell \to \R^d$ such that $f$ and its derivatives up to order $2$ are uniformly bounded; that is
\begin{equation*}
C^{2}_b(\R^\ell;\R^d) := \{f \in C^2(\R^\ell;\R^d) : \| f\|_{C^{2}_b} < \infty\}
\end{equation*}
with
\begin{equation*}
\|f\|_{C^{2}_b} := \|f\|_{\infty} + \|\D f\|_{\infty} + \|\D^{2}f\|_{\infty}.
\end{equation*}

For $S \in \Omega_p$ we introduce the set $\mathcal{G}^{2}_S$ of all generalized functionally generated trading strategies $\phi^f$, which are all strategies of the form
\begin{equation}\label{eq:phi generated strategy}
  \phi_t^f = (f^1_t,\dots,f^{d}_t) := f(S_t,A_t), \qquad t \in [0,\infty),
\end{equation}
for some $f \in C^{2}_b(\R^\ell;\R^d)$. For $\phi^f \in \cG^2_S$ the corresponding capital process is given by
\begin{equation}\label{eq:capital process}
  V^{f}_t(S) = \lim_{n \to \infty} \sum_{k=0}^{N_n-1} \sum_{i=1}^d f^i_{t^n_k}(S^i_{t^n_{k+1}\wedge t}-S^i_{t^n_{k}\wedge t}), \qquad t \in [0,\infty).
\end{equation}

\begin{proposition}\label{prop: capital process well-defined}
  Let $p \in (2,3)$ and $S \in \Omega_p$, and let $\phi^f, \phi^{\tf} \in \cG^2_S$. Then, $\phi^f\in \mathcal{A}_S$ is an admissible strategy, and the capital process $(V^f_t(S))_{t\in [0,\infty)}$ given in \eqref{eq:capital process} is well-defined as a locally uniform limit in $t \in [0,\infty)$. Moreover, for every $T \in [0,\infty)$, we have the stability estimate
  \begin{equation}\label{eq:capital stability estimate}
    |V^{f}_T(S) - V^{\tf}_T(S)| \leq C \|f - \tf\|_{C^{2}_b} (1 + \|S\|_{p,[0,T]}^2 + \|A\|_{1,[0,T]}) (1 + \|S\|_{p,[0,T]}) \|\bS\|_{p,[0,T]},
  \end{equation}
  where the constant $C$ depends only on $p$, and the triplet $\bS = (S,S,A)$ is the c{\`a}dl{\`a}g rough path defined in Lemma~\ref{lem: SSA is rough path}.
\end{proposition}

\begin{proof}
  \textit{Admissibility:} Let $\phi = \phi^f \in \cG^2_S$ be a functionally generated strategy $\phi=(\phi^1,\dots,\phi^d)$ of the form in \eqref{eq:phi generated strategy} for some $f \in C^{2}_b(\R^\ell;\R^d)$. Fix a $T \in [0,\infty)$. We claim that $(\phi,\phi') \in \cV^{p,\frac{p}{2}}_S$ is a controlled path with respect to $S$ in the sense of Definition~\ref{def: controlled path} (with $q = p$ and $r = p/2$), where
  $$
    \phi'_t := \D_S f(S_t,A_t), \qquad t \in [0,T],
  $$
  and $\D_Sf$ denotes the derivative of $f$ with respect to its first $d$ components. To see this, we first note that
  \begin{equation*}
    |\phi'_{s,t}| = |\D_S f(S_t,A_t) - \D_S f(S_s,A_s)| \leq \|f\|_{C^2_b} (|S_{s,t}| + |A_{s,t}|),
  \end{equation*}
  so that
  \begin{equation}\label{eq:phi' bound}
    \|\phi'\|_{p,[0,T]} \lesssim \|f\|_{C^2_b} (\|S\|_{p,[0,T]} + \|A\|_{1,[0,T]}) < \infty,
  \end{equation}
  and hence $\phi' \in D^p([0,T];\cL(\R^d;\R^d))$. We moreover have that
  \begin{align*}
    R^\phi_{s,t} &:= \phi_{s,t} - \phi'_s S_{s,t}\\
    &= f(S_t,A_t) - f(S_s,A_s) - \D_S f(S_s,A_s) S_{s,t}\\
    &= f(S_t,A_s) - f(S_s,A_s) - \D_S f(S_s,A_s) S_{s,t} + f(S_t,A_t) - f(S_t,A_s)\\
    &= \int_0^1 \Big(\D_S f(S_s + \tau S_{s,t},A_s) - \D_S f(S_s,A_s)\Big) S_{s,t} \dd \tau + f(S_t,A_t) - f(S_t,A_s),
  \end{align*}
  so that $|R^\phi_{s,t}| \leq \|f\|_{C^{2}_b} (|S_{s,t}|^2 + |A_{s,t}|)$. It follows that
  \begin{equation}\label{eq:R^phi bound}
    \|R^\phi\|_{\frac{p}{2},[0,T]} \lesssim \|f\|_{C^{2}_b} (\|S\|_{p,[0,T]}^2 + \|A\|_{1,[0,T]}) < \infty,
  \end{equation}
  so that $R^\phi \in D^{p/2}(\Delta_{[0,T]};\R^d)$, and thus the conditions of Definition~\ref{def: controlled path} are satisfied. Thus, by Proposition~\ref{prop: general rough integral exists} (and Remark~\ref{remark classical rough integral as special case}), we have the existence for each $t \in [0,T]$ of the ($\R$-valued) rough integral
  $$
    \int_0^t \phi_s \dd \bS_s = \lim_{|\cP| \to 0} \sum_{[u,v] \in \cP} \phi_u S_{u,v} + \phi'_u A_{u,v}.
  $$

For a given path $F$, let $J_F = \{t \in (0,T] : F_{t-,t} \neq 0\}$ denote the jump times of $F$. It follows from Property~\textup{(RIE)} and Proposition~\ref{prop: uniform convergence} that $J_S \subseteq \cup_{n \in \N} \cP^n$. Since we also assumed that $J_A \subseteq \cup_{n \in \N} \cP^n$, it then follows from \eqref{eq:phi generated strategy} that $J_\phi \subseteq \cup_{n \in \N} \cP^n$. Thus, by Theorem~\ref{thm: rough int as limit Riemann sums}, we have that
  \begin{equation*}
    \int_0^t \phi_s \dd \bS_s = \lim_{n \to \infty} \sum_{k=0}^{N_n-1} \phi_{t^n_k} (S_{t^n_{k+1} \wedge t} - S_{t^n_{k} \wedge t}) = \lim_{n \to \infty} \sum_{k=0}^{N_n-1} \sum_{i=1}^d f^i_{t^n_k} (S^i_{t^n_{k+1} \wedge t} - S^i_{t^n_{k} \wedge t}) = V^f_t(S),
  \end{equation*}
  and that this limit is uniform in $t \in [0,T]$.

  \textit{Stability estimate:} Let $\phi$ and $\tphi$ be the strategies generated by $f$ and $\tf$ respectively, as defined in \eqref{eq:phi generated strategy}. By part~(i) of Proposition~\ref{prop general rough int is cts}, we have the estimate
  \begin{align}\label{eq:estimate phi phi' R^phi}
    \begin{split}
    &|V^{f}_T(S) - V^{\tf}_T(S)|\\
    &= \bigg|\int_0^T \phi_s \dd \bS_s - \int_0^T \tphi_s \dd \bS_s\bigg| \leq \bigg\|\int_0^\cdot \phi_s \dd \bS_s - \int_0^\cdot \tphi_s \dd \bS_s\bigg\|_{p,[0,T]}\\
    &\lesssim (|\phi_0 - \tphi_0| + |\phi'_0 - \tphi'_0| + \|\phi' - \tphi'\|_{p,[0,T]} + \|R^\phi - R^{\tphi}\|_{\frac{p}{2},[0,T]}) (1 + \|S\|_{p,[0,T]}) \|\bS\|_{p,[0,T]}.
    \end{split}
  \end{align}
  As above, here
  $$
    \phi'_t = \D_S f(S_t,A_t) \qquad \text{and} \qquad R^\phi_{s,t} = \phi_{s,t} - \phi'_s S_{s,t},
  $$
  with $\tphi'$ and $R^{\tphi}$ defined similarly. We will now aim to estimate each term on the right-hand side of~\eqref{eq:estimate phi phi' R^phi}.

  We have that
  \begin{equation}\label{eq:est phi 1}
    |\phi_0 - \tphi_0| = |f(S_0,A_0) - \tf(S_0,A_0)| \leq \|f - \tf\|_\infty \leq \|f - \tf\|_{C^{2}_b},
  \end{equation}
  and similarly
  \begin{equation}\label{eq:est phi 2}
    |\phi'_0 - \tphi'_0| = |\D_S f(S_0,A_0) - \D_S \tf(S_0,A_0)| 
    \leq \|\D_S f - \D_S \tf\|_\infty \leq \|f - \tf\|_{C^{2}_b}.
  \end{equation}
  Next, for $[s,t] \subseteq [0,T]$, we compute
  \begin{align*}
    (\phi' - \tphi')_{s,t}
    &= \D_S f(S_t,A_t) - \D_S f(S_s,A_s) - \D_S \tf(S_t,A_t) + \D_S \tf(S_s,A_s)\\
    &= \D_S f(S_t,A_s) - \D_S f(S_s,A_s) - \D_S \tf(S_t,A_s) + \D_S \tf(S_s,A_s)\\
    &\quad + \D_S f(S_t,A_t) - \D_S f(S_t,A_s) - \D_S \tf(S_t,A_t) + \D_S \tf(S_t,A_s)\\
    &= \int_0^1 \Big(\D^2_{SS} f(S_s + \tau S_{s,t},A_s) - \D^2_{SS} \tf(S_s + \tau S_{s,t},A_s)\Big) S_{s,t} \dd \tau\\
    &\quad + \int_0^1 \Big(\D^2_{SA}f(S_t,A_s + \tau A_{s,t}) - \D^2_{SA} \tf(S_t,A_s + \tau A_{s,t})\Big) A_{s,t} \dd \tau,
  \end{align*}
  so that
  \begin{equation*}
    |(\phi' - \tphi')_{s,t}| 
    \leq \|\D^2_{SS}f - \D^2_{SS}\tf\|_\infty |S_{s,t}| + \|\D^2_{SA}f - \D^2_{SA}\tf\|_\infty |A_{s,t}|
    \leq \|f - \tf\|_{C^{2}_b} (|S_{s,t}| + |A_{s,t}|),
  \end{equation*}
  and thus
  \begin{equation}\label{eq:est phi 3}
    \|\phi' - \tphi'\|_{p,[0,T]} \lesssim \|f - \tf\|_{C^{2}_b} (\|S\|_{p,[0,T]} + \|A\|_{1,[0,T]}).
  \end{equation}
  Finally, we have that
  \begin{align*}
    &(R^{\phi} - R^{\tphi})_{s,t}\\
    &= f(S_t,A_t) - f(S_s,A_s) - \D_S f(S_s,A_s) S_{s,t} - \tf(S_t,A_t) + \tf(S_s,A_s) + \D_S \tf(S_s,A_s) S_{s,t}\\
    &= f(S_t,A_s) - f(S_s,A_s) - \D_S f(S_s,A_s) S_{s,t} - \tf(S_t,A_s) + \tf(S_s,A_s) + \D_S \tf(S_s,A_s) S_{s,t}\\
    &\quad + f(S_t,A_t) - f(S_t,A_s) - \tf(S_t,A_t) + \tf(S_t,A_s)\\
    &= \int_0^1 \int_0^1 \Big(\D^2_{SS}f(S_s + \tau_1 \tau_2 S_{s,t},A_s) - \D^2_{SS}\tf(S_s + \tau_1 \tau_2 S_{s,t},A_s)\Big) S_{s,t}^{\otimes 2} \tau_1 \dd \tau_2 \dd \tau_1\\
    &\quad + \int_0^1 \Big(\D_A f(S_t,A_s + \tau A_{s,t}) - \D_A \tf(S_t,A_s + \tau A_{s,t})\Big) A_{s,t} \dd \tau,
  \end{align*}
  so that
  \begin{equation*}
    |(R^{\phi} - R^{\tphi})_{s,t}| 
    \leq \|\D^2_{SS}f - \D^2_{SS}\tf\|_\infty |S_{s,t}|^2 + \|\D_A f - \D_A \tf\|_\infty |A_{s,t}| \leq \|f - \tf\|_{C^2_b} (|S_{s,t}|^2 + |A_{s,t}|),
  \end{equation*}
  and hence
  \begin{equation}\label{eq:est phi 4}
    \|R^{\phi} - R^{\tphi}\|_{\frac{p}{2},[0,T]} \lesssim \|f - \tf\|_{C^2_b} (\|S\|_{p,[0,T]}^2 + \|A\|_{1,[0,T]}).
  \end{equation}
  Substituting \eqref{eq:est phi 1}, \eqref{eq:est phi 2}, \eqref{eq:est phi 3} and \eqref{eq:est phi 4} into \eqref{eq:estimate phi phi' R^phi}, we deduce that the estimate in \eqref{eq:capital stability estimate} holds.
\end{proof}

\subsection{Path-dependent functionally generated trading strategies}

The functionally generated trading strategies considered in Section~\ref{subsec:functionally generated trading strategies} could depend on the past prices only through a process of locally finite variation. In some contexts it is beneficial to work with trading strategies possessing a more general path-dependent structure; see e.g.~Schied, Speiser and Voloshchenko \cite{Schied2016,Schied2018} for more detailed discussions in this direction. A common way to treat path-dependent and non-anticipating trading strategies is the calculus initiated by Dupire \cite{Dupire2019} and Cont and Fourni\'e \cite{Cont2010}. For the sake of brevity we recall here only the essential definitions and refer to Ananova \cite[Section~3.1]{Ananova2020} and \cite{Cont2010} for full details.

A functional $F \colon [0,T] \times D([0,T];\mathbb{R}^d) \rightarrow \mathbb{R}^d$ is called non-anticipative if
\begin{equation*}
F(t,S) = F(t,S_{\cdot \wedge t}) \quad \text{for all} \quad S \in D([0,T];\mathbb{R}^d).
\end{equation*}
As usual, the non-anticipative functionals are defined on the space of stopped paths, defined as the equivalence class in $[0,T] \times D([0,T];\mathbb{R}^d)$ with respect to the following equivalence relation:
\begin{equation*}
(t,S) \sim (t',S') \Longleftrightarrow t = t' \text{ and } S_{\cdot \wedge t} = S'_{\cdot \wedge t'}.
\end{equation*}
The resulting space $\Lambda_T^d$ is equipped with the distance
\begin{equation*}
d_\infty((t,S),(t',S')) := |t - t'| + \sup_{u \in [0, T]} |S_{u \wedge t} - S'_{u \wedge t'}|,
\end{equation*}
which turns $(\Lambda_T^d, d_\infty)$ into a complete metric space. We introduce the following spaces:
\begin{itemize}
\item $\mathbb{C}^{0,0}_l(\Lambda_T^d)$ is the space of left-continuous functionals $F \colon \Lambda_T^d \to \R^d$, i.e.~for all $(t,S) \in \Lambda_T^d$ and $\epsilon > 0$, there exists $\nu > 0$ such that, for all $(t', {S}') \in \Lambda_T^d$,
\begin{equation*}
t' < t \text{ and } d_\infty((t,S),(t',S')) < \nu \Longrightarrow |F(t,S) - F(t',S')| < \epsilon.
\end{equation*}
\item $\mathbb{B}(\Lambda_T^d)$ is the space of all boundedness-preserving functionals $F \colon \Lambda_T^d \to \R^d$, i.e.~for every compact subset $K \subset \mathbb{R}^d$ and for every $t_0 \in [0,T]$, there exists $C > 0$ such that, for all $t \in [0,t_0]$ and $(t,S) \in \Lambda_T^d$,
\begin{align*}
S([0,t]) \subseteq K \Longrightarrow |F(t,S)| < C.
\end{align*}
\item $\mathrm{Lip}(\Lambda_T^d,d_{\infty})$ is the space of all Lipschitz continuous functionals $F \colon \Lambda_T^d \to \R^d$, i.e.~there exists $C > 0$ such that, for all $(t,S), (t',S') \in \Lambda^d_T$,
\[|F(t,S) - F(t',S')| \leq C d_{\infty}((t,S),(t',S')).\]
\end{itemize}

We define $\mathbb{C}^{1,1}_b(\Lambda_T^d)$ as the set of non-anticipative functionals $F \colon \Lambda_T^d \to \mathbb{R}$ which are:
\begin{itemize}
\item horizontally differentiable, i.e.~for all $(t,S) \in \Lambda_T^d$,
\begin{equation*}
\mathcal{D}F(t,S) = \lim_{h \downarrow 0} \frac{F(t + h,S_{\cdot \wedge t}) - F(t,S_{\cdot \wedge t})}{h}
\end{equation*}
exists, and $\mathcal{D}F$ is continuous at fixed times,
\item and vertically differentiable, i.e.~for all $(t,S) \in \Lambda_T^d$, $\nabla_{x} F(t,S) = (\partial_i F(t,S))_{i = 1, \dots, d}$, with
\begin{equation*}
\partial_i F(t,S) = \lim_{h \rightarrow 0} \frac{F(t,S_{\cdot \wedge t} + h e_i \1_{[t,T]}) - F(t,S_{\cdot \wedge t})}{h},
\end{equation*}
exists, where $(e_i)_{i = 1, \ldots, d}$ is the canonical basis of $\mathbb{R}^d$, and $\nabla_{x} F \in \mathbb{C}_l^{0,0}(\Lambda_T^d)$,
\item and such that $\mathcal{D}F, \nabla_{x} F \in \mathbb{B}(\Lambda_T^d)$.
\end{itemize}

\begin{corollary}
If $F \in \mathbb{C}^{1,1}_b(\Lambda^d_T)$ with $F$ and $\nabla_{x} F$ in $\mathrm{Lip}(\Lambda^d_T, d_{\infty})$, and $S \in \Omega_p$, then the path-dependent functionally generated trading strategy $F(\cdot,S)$ is an admissible strategy in the sense of Definition~\ref{def: admissible strategy}.
\end{corollary}

\begin{proof}
That $(F(\cdot,S),\nabla_{x} F(\cdot,S))$ is a controlled path with respect to $S$ is an immediate consequence of \cite[Lemma~5.12]{Ananova2019}; see also \cite[Lemma~3.7]{Ananova2020}. The admissibility condition regarding the jump times of $F(\cdot,S)$ is ensured by the Lipschitz continuity of $F$.
\end{proof}

\begin{remark}
The standard examples of sufficiently regular path-dependent functionals are functionals which depend on the running maximum or on a notion of the average of the underlying path; see e.g.~Ananova \cite{Ananova2019}. Further examples include:
\begin{enumerate}[(i)]
\item $F(t,S_{\cdot \wedge t}) := \int_0^t \psi(s,S_{\cdot \wedge s-}) \dd [S]_s$,
\item $F(t,S_{\cdot \wedge t}) := \int_0^t \nabla_x f(s,S_{\cdot \wedge s}) \dd S_s$,
\item $F(t,S_{\cdot \wedge t}) := \sum_{i=1}^d \int_0^t (S_t^i - S_s^i) f_i(S^i_s) \dd S^i_s - \int^t_0 f_i(S^i_s) \dd [S^i]_s$,
\end{enumerate}
for a twice differentiable function $f = (f_1, \ldots, f_d) \colon \R^d \to \R^d$ and a left-continuous and locally bounded function $\psi \colon \Lambda_T^d \to \R^{d \times d}$; see Chiu and Cont \cite[Example~4.18]{Chiu2022}.
\end{remark}

\subsection{Cover's universal portfolio}

While functionally generated trading strategies are most prominent in the literature regarding hedging and control problems in mathematical finance, various other trading strategies with desirable properties have been considered. One example coming from portfolio theory is Cover's universal portfolio, as introduced in \cite{Cover1991}. The basic idea is to invest, not according to one specific trading strategy, but according to a mixture of all admissible strategies. Following Cuchiero, Schachermayer and Wong \cite{Cuchiero2019}, we introduce here a model-free analogue of Cover's universal portfolio.

\smallskip

Let $\mathcal{Z}$ be a Borel measurable subset of $C^2_b(\R^d;\R^d)$, and suppose that $\nu$ is a probability measure on $\mathcal{Z}$. A model-free version of Cover's universal portfolio $\phi^{\nu}$ is then given by
\begin{equation}\label{eq:defn Cover's portfolio}
  \phi^{\nu}_t := \int_{\mathcal{Z}} \phi^f_t \dd \nu(f), \qquad t \in [0,\infty),
\end{equation}
where $\phi^f = f(S)$ is the portfolio generated by $f$, for some fixed $S \in \Omega_p$.

\begin{lemma}
  Let $S \in \Omega_p$ and let $\nu$ be a probability measure on $\mathcal{Z}$, as above. If
  $$
    \int_{\mathcal{Z}} \|f\|_{C^2_b} \dd \nu(f) < \infty,
  $$
  then Cover's universal portfolio $\phi^{\nu}$, as defined in \eqref{eq:defn Cover's portfolio}, is an admissible strategy in the sense of Definition~\ref{def: admissible strategy}.
\end{lemma}

\begin{proof}
  Let $T > 0$. We know from Proposition~\ref{prop: capital process well-defined} that, for each $f \in \mathcal{Z}$, the corresponding functionally generated portfolio $\phi = \phi^f$ is an admissible strategy. Let $\phi'$ and $R^\phi$ denote the corresponding Gubinelli derivative and remainder term. It follows from the inequalities in \eqref{eq:phi' bound} and \eqref{eq:R^phi bound} that
  \begin{equation*}
   |\phi_0| + \|\phi,\phi'\|_{\cV^{p,p/2}_S} = |\phi_0| + |\phi'_0| + \|\phi'\|_{p,[0,T]} + \|R^\phi\|_{\frac{p}{2},[0,T]} \lesssim \|f\|_{C^{2}_b} (1 + \|S\|_{p,[0,T]}^2),
  \end{equation*}
  and hence that
  \begin{equation}\label{eq:integrability for Cover}
    \int_{\mathcal{Z}} (|\phi_0| + \|\phi,\phi'\|_{\cV^{p,p/2}_S}) \dd \nu \lesssim (1 + \|S\|_{p,[0,T]}^2) \int_{\mathcal{Z}} \|f\|_{C^{2}_b} \dd \nu < \infty.
  \end{equation}

  Recall that the map $(\phi,\phi') \mapsto |\phi_0| + \|\phi,\phi'\|_{\cV^{p,p/2}_S}$ is a norm on the Banach space of controlled paths $\cV_S^{p,\frac{p}{2}}$. Note moreover that the subset of controlled paths $(\phi,\phi') \in \cV_S^{p,\frac{p}{2}}$ satisfying $J_\phi \subseteq \cup_{n \in \N} \cP^n$ is a closed linear subspace of $\cV_S^{p,\frac{p}{2}}$, and thus is itself a Banach space.

  It follows from the integrability condition in \eqref{eq:integrability for Cover} that the integral in \eqref{eq:defn Cover's portfolio} exists as a well-defined Bochner integral, and defines a controlled path $(\phi^\nu,(\phi^\nu)') \in \cV_S^{p,\frac{p}{2}}$ satisfying $J_{\phi^\nu} \subseteq \cup_{n \in \N} \cP^n$.
\end{proof}

\subsection{Functionally generated portfolios from stochastic portfolio theory}

In this section we briefly discuss admissible strategies appearing in stochastic portfolio theory, as initiated by Fernholz \cite{Fernholz2002}; see also e.g.~Strong \cite{Strong2014} and Karatzas and Ruf \cite{Karatzas2017}. Following the model-free framework for stochastic portfolio theory, as introduced in Schied, Speiser and Voloshchenko \cite{Schied2018} and Cuchiero, Schachermayer and Wong \cite{Cuchiero2019}, we consider a $d$-dimensional c\`adl\`ag path $S = (S^1, \ldots, S^d)$ such that $S^i_t > 0$ for all $t \geq 0$ and $i = 1, \ldots, d$. The total capitalization $\Sigma$ is defined by $\Sigma_t := S^1_t + \cdots + S^d_t$, and the relative market weight process $\mu$ is given by
\[\mu^i_t := \frac{S^i_t}{\Sigma_t} = \frac{S^i_t}{S^1_t + \cdots + S^d_t}, \qquad i = 1, \ldots, d,\]
which takes values in the open unit simplex $\Delta^d_+ := \{x \in \R^d : \sum_{i=1}^d x_i = 1, x_i > 0 \text{ for all } i\}$.

Here we impose that the market weight process $\mu = (\mu^1, \ldots, \mu^d)$ satisfies Property \textup{(RIE)} with respect to some $p \in (2,3)$ and a sequence of nested partitions $(\cP^n)_{n \in \N}$. More precisely, we assume that $\mu$ is a price path, in the sense of Definition~\ref{defn price path}.

\smallskip

Given a controlled path $\theta \in \crpmuq$, we define its associated value evolution $V^\theta$ by $V^\theta_t := \sum_{i=1}^d \theta^i_t \mu^i_t$ for $t \geq 0$. We also denote $Q^\theta_t := V^\theta_t - V^\theta_0 - \int_0^t \theta_s \dd \mu_s$, where $\int_0^t \theta_s \dd \mu_s$ is interpreted as a rough integral, as in Remark~\ref{remark classical rough integral as special case}. As in classical mathematical finance (see e.g.~\cite[Proposition~2.3]{Karatzas2017}) we can show that every controlled path $\theta \in \crpmuq$ induces a self-financing trading strategy $\varphi$, and that every such strategy $\varphi$ is itself a controlled path in $\crpmuq$.

\begin{proposition}
Given $\theta \in \crpmuq$ and a constant $C \in \R$, we introduce
\[\varphi^i_t := \theta^i_t - Q^\theta_t - C, \qquad t \geq 0, \quad i = 1, \ldots, d.\]
Then the resulting path $\varphi = (\varphi^1, \ldots, \varphi^d)$ is a controlled path in $\crpmuq$, and $\varphi$ is self-financing, in the sense that
\[V^\varphi_t - V^\varphi_0 = \int_0^t \varphi_s \dd \mu_s, \qquad t \geq 0,\]
where $V^\varphi_t = \sum_{i=1}^d \varphi^i_t \mu^i_t$. Moreover, if $\theta$ is an admissible strategy in the sense of Definition~\ref{def: admissible strategy}, then so is $\varphi$.
\end{proposition}

\begin{proof}
Since $\theta \in \crpmuq$ is a controlled path, by (the trivial extension to c\`adl\`ag paths of) \cite[Lemma~A.1]{Allan2023}, the path $V^\theta = \sum_{i=1}^d \theta^i \mu^i$ is also a controlled path. Recalling Remark~\ref{rmk: rough int is controlled path}, we also have that the rough integral $\int_0^\cdot \theta \dd \mu$ is itself a controlled path (with Gubinelli derivative equal to $\theta$). It is then clear that $Q^\theta = V^\theta - V^\theta_0 - \int_0^\cdot \theta_s \dd \mu_s$ is a controlled path, and hence that $\varphi^i = \theta^i - Q^\theta - C$ is as well for every $i = 1, \ldots, d$, so that $\varphi \in \crpmuq$. The self-financing property of $\varphi$ can be verified by following the proof of \cite[Proposition~2.3]{Karatzas2017}.

We know from Proposition~\ref{prop: uniform convergence} and the fact that $\mu$ satisfies Property \textup{(RIE)} that the jump times $J_\mu$ of $\mu$ satisfy $J_\mu \subseteq \cup_{n \in \N} \cP^n$. It is also clear that the jump times of the integral $\int_0^\cdot \theta \dd \mu$ form a subset of $J_\mu$. Thus, under the assumption that $\theta$ is an admissible strategy, so that $J_\theta \subseteq \cup_{n \in \N} \cP^n$, we also deduce that $J_\varphi \subseteq \cup_{n \in \N} \cP^n$, so that $\varphi$ is itself an admissible strategy.
\end{proof}

The following definition introduces notions which can be seen as the rough path counterparts of the functionally generated strategies induced by regular and Lyapunov functions from stochastic portfolio theory; see \cite[Definitions~3.1 and 3.3]{Karatzas2017}.

\begin{definition}\label{def: regular function}
We say that a continuous function $G \colon \Delta^d_+ \to \R$ is \emph{regular} for the market weight process $\mu$ if
\begin{enumerate}[(i)]
\item there exists a measurable function $\D G = (\D_1 G, \ldots, \D_d G)$ on $\Delta^d_+$ such that the path $\theta = (\theta^1, \ldots, \theta^d)$, given by $\theta^i = \D_i G(\mu)$ for $i = 1, \ldots, d$, is an admissible strategy (in the sense of Definition~\ref{def: admissible strategy}), and
\item the (c\`adl\`ag) path $\Gamma^G$, given by
\[\Gamma^G_t := G(\mu_0) - G(\mu_t) + \int_0^t \theta_s \dd \mu_s, \qquad t \geq 0,\]
has locally bounded variation.
\end{enumerate}

We say that a regular function $G \colon \Delta^d_+ \to \R$ is a \emph{Lyapunov} function for the market weight process $\mu$ if the path $\Gamma^G$ is also nondecreasing.
\end{definition}

\begin{example}
Suppose that $G \colon \Delta^d_+ \to \R$ is a $C^3$ function. It is then straightforward to see that $\theta = \D G(\mu)$ defines a controlled path, with $\D G$ here defined as the classical gradient of $G$. It also follows from the It\^o formula for rough paths (recall Remark~\ref{remark; Ito formula}) that
\begin{equation}\label{eq:Gamma in smooth G example}
\Gamma^G_t = -\frac{1}{2} \int_0^t \D^2 G(\mu_s) \dd [\mu]^c_s - \sum_{s \leq t} (G(\mu_s) - G(\mu_{s-}) - \D G(\mu_{s-}) \Delta \mu_s),
\end{equation}
where $[\mu]^c$ is the continuous part of the quadratic variation of $\mu$, and $\Delta \mu_s = \mu_s - \mu_{s-}$ denotes the jump of $\mu$ at time $s$.

If we also assume that $G$ is concave, then we infer from \eqref{eq:Gamma in smooth G example} that $\Gamma^G$ is nondecreasing. In this case $G$ is a Lyapunov function in the sense of Definition~\ref{def: regular function}. Two important examples of such functions $G$ are the Gibbs entropy function, $H(x) := \sum_{i=1}^d x_i \log \frac{1}{x_i}$, and the quadratic function, $Q^{(c)}(x):= c - \sum_{i=1}^d x_i^2$ for some $c \in \R$; see \cite[Section~5.1]{Karatzas2017}.
\end{example}

\begin{remark}
If $\mu$ is realized by a semimartingale model, then any continuous concave (but not necessarily $C^3$) function $G$ can be a candidate for a Lyapunov function in the sense of \cite[Definition~3.3]{Karatzas2017}. This is essentially because in stochastic portfolio theory one only needs the ``supergradients'' $\D G$ to be measurable, so that $\D G(\mu)$ is integrable with respect to the semimartingale $\mu$. In contrast, in our purely pathwise setup, measurability of $\D G$ alone is not enough to ensure that $\theta = \D G(\mu)$ is controlled by $\mu$, so we require more regularity of the generating function $G$. This illustrates a difference between stochastic integration in a probabilistic setting and rough integration when only a single deterministic path is considered; for more detailed discussions on this theme, we refer to \cite{Allan2023}. On the other hand, as shown in the previous example, many important Lyapunov functions from stochastic portfolio theory (such as the Gibbs entropy function) are actually smooth, so they do induce controlled paths (and indeed admissible strategies in the sense of Definition~\ref{def: admissible strategy}) for almost all trajectories of $\mu$, and the stability results established in Section~\ref{sec:trading strategies} remain valid for these functionally generated strategies. It would be interesting to explore further Lyapunov functions for rough paths, but this is beyond the scope of the present paper.
\end{remark}

Based on the previous observations, we can also extend the following notions from classical stochastic portfolio theory (e.g.~\cite{Karatzas2017}) to our rough path setting. The proof is straightforward and is therefore omitted for brevity.

\begin{corollary}
Suppose that the market weight process $\mu = (\mu^1, \ldots, \mu^d)$ is a price path in the sense of Definition~\ref{defn price path}, and let $\theta = (\theta^1, \ldots, \theta^d)$ be an admissible strategy for $\mu$ in the sense of Definition~\ref{def: admissible strategy}. Let $G$ be a regular function for $\mu$ in the sense of Definition~\ref{def: regular function}. Then the following paths are also admissible strategies for $\mu$:
\begin{enumerate}[(i)]
\item the additive generated strategy: $\varphi^i_t = \theta^i_t - Q^\theta_t - C_0$, $i = 1, \ldots, d$, where $Q^\theta_t = V^\theta_t - V^\theta_0 - \int_0^t \theta_s \dd \mu_s$ and $C_0 = \sum_{j=1}^d \mu^i_0 \D_j G(\mu_0) - G(\mu_0)$,
\item the portfolio weights associated with $\varphi$: $\pi^i_t = \frac{\mu^i_t \varphi^i_t}{\sum_{j=1}^d \mu^j_t \varphi^j_t}$, $i = 1, \ldots, d$,
\item the multiplicative generated strategy: $\Phi^i_t = \eta^i_t - Q^\eta_t - C_0$, for $\eta^i_t = \theta^i_t \exp(\int_0^t \frac{\dd \Gamma^G_s}{G(\mu_s)})$, $i = 1, \ldots, d$, and $Q^\eta_t = V^\eta_t - V^\eta_0 - \int_0^t \eta_s \dd \mu_s$, where here the function $G$ is also assumed to be positive and bounded away from zero,
\item the portfolio weights associated with $\Phi$: $\Pi^i_t = \mu^i_t(1 + \frac{1}{G(\mu_t)} (\D_i G(\mu_t) - \sum_{j=1}^d \D_j G(\mu_t) \mu^j_t))$, $i = 1, \ldots, d$.
\end{enumerate}
\end{corollary}

\section{Semimartingales and typical price paths satisfy Property~(RIE)}\label{sec:stochastic processes}

In this section we show that many stochastic processes commonly used to model the price evolutions on financial markets satisfy Property~\textup{(RIE)} along a suitable sequence of partitions. In particular, we verify that Property~\textup{(RIE)} holds for semimartingales and for typical price paths in the sense of Vovk \cite{Vovk2012}.

\subsection{Semimartingales}\label{subsect: semimartingale}

Semimartingales, such as geometric Brownian motion and Markov jump-diffusion processes, serve as the most frequently used stochastic processes to model price evolutions on financial markets. For more details on semimartingales and It{\^o} integration we refer to the standard textbook by Protter \cite[Chapter~II]{Protter2005}.

Usually the considered class of semimartingales is restricted to those satisfying the condition ``no free lunch with vanishing risk'' in classical mathematical finance, e.g.~Delbaen and Schachermayer \cite{Delbaen1994}, or the condition ``no unbounded profit with bounded risk'' (NUPBR) in stochastic portfolio theory; see e.g.~Karatzas and Kardaras \cite{Karatzas2007}. Such a restriction is not required here. Property~\textup{(RIE)} is fulfilled by general c{\`a}dl{\`a}g semimartingales with respect to any $p \in (2,3)$ and a suitable (random) sequence of partitions.

\smallskip

Let us fix a filtered probability space $(\Omega,\mathcal{F},(\mathcal{F}_t)_{t \in [0,\infty)},\P)$ and assume that the filtration $(\mathcal{F}_t)_{t \in [0,\infty)}$ satisfies the usual conditions, i.e.~completeness and right-continuity. On this probability space we consider a $d$-dimensional c{\`a}dl{\`a}g semimartingale $X = (X_t)_{t \in [0,\infty)}$.

\smallskip

For each $n \in \N$, we introduce stopping times $(\tau^n_k)_{k \in \N \cup \{0\}}$ such that $\tau^n_0 = 0$, and
$$
  \tau^n_k := \inf \{t > \tau^n_{k-1} : |X_t - X_{\tau^n_{k-1}}| \geq 2^{-n}\}, \qquad k \in \N.
$$
We then define a sequence of partitions $(\cP_X^n)_{n \in \N}$ by
$$
  \cP_X^n := \{\tau^m_k : m \leq n, k \in \N \cup \{0\}\}.
$$
Note that $(\cP_X^n)_{n \in \N}$ is a nested sequence of adapted partitions. However, this sequence of partitions will not have vanishing mesh size if the path $X$ has an interval of constancy. To amend this, we proceed as follows. For each $n \in \N$ and $k \in \N \cup \{0\}$, we set
\begin{align*}
\sigma^n_k &:= \tau^n_{k+1} \wedge \inf \{t > \tau^n_k : \exists \delta > 0 \text{ such that } X_s = X_t \text{ for all } s \in [t,t + \delta]\},\\
\varsigma^n_k &:= \tau^n_{k+1} \wedge \inf \{t > \sigma^n_k: X_t \neq X_{\sigma^n_k}\},
\end{align*}
that is, the beginning and end points of the first interval of constancy of $X$ within the interval $[\tau^n_k,\tau^n_{k+1}]$. For each $i \in \N$, we then define $\rho^n_{k,i} := (\sigma^n_k + i 2^{-n}) \wedge \varsigma^n_k$. Clearly, for each $n, k$, we will have that $\rho^n_{k,i} = \varsigma^n_k$ for all but finitely many $i \in \N$ (provided that $\varsigma^n_k < \infty$). Finally, we define
\begin{equation*}
\mathcal{Q}^n_X := \cP_X^n \cup \{\sigma^m_k : m \leq n, k \in \N \cup \{0\}\} \cup \{\rho^m_{k,i} : m \leq n, k \in \N \cup \{0\}, i \in \N\}.
\end{equation*}
The sequence of partitions $(\mathcal{Q}_X^n)_{n \in \N}$ is still nested, and moreover has vanishing mesh size on every compact time interval. Moreover, it is straightforward to see that all the points $\tau^n_k$, $\sigma^n_k$ and $\rho^n_{k,i}$ appearing in the partitions $(\mathcal{Q}_X^n)_{n \in \N}$ are stopping times with respect to the (right-continuous) filtration $(\mathcal{F}_t)_{t \in [0,\infty)}$. In the next proposition we will show that $X$ satisfies Property~\textup{(RIE)} with respect to any $p \in (2,3)$ and the sequence of partitions $(\mathcal{Q}^n_X)_{n \in \N}$.

\begin{proposition}\label{prop: semimartingales satisfy RIE}
  Let $p \in (2,3)$, and let $X$ be a $d$-dimensional c{\`a}dl{\`a}g semimartingale. Then almost every sample path of $X$ is a price path in the sense of Definition~\ref{defn price path}.
\end{proposition}

\begin{proof}
\textit{Step 1.}
Fix a $T > 0$. Let $\int_0^\cdot X_{u-} \otimes \d X_u$ denote the It{\^o} integral of $X$ with respect to itself. For each $n \in \N$, we define the discretized process $X^n = (X^n_t)_{t \in [0,T]}$ by
  \begin{equation}\label{eq:defn X^n}
    X^n_t := \sum_{[u,v]\in \mathcal{Q}^n_X([0,T])} X_{u} \1_{[u,v)}(t), \qquad t \in [0,T],
  \end{equation}
  so that in particular $\int_0^t X^n_{u-} \otimes \d X_u = \sum_{[u,v] \in \mathcal{Q}^n_X([0,T])} X_{u} \otimes X_{u \wedge t, v \wedge t}$ for $t \in [0,T]$. By the definition of the partition $\mathcal{Q}^n_X$, we have that
  $$
    \|X^n_{\cdot -} - X_{\cdot -}\|_{\infty,[0,T]} \leq 2^{1 - n} \qquad \text{for all} \qquad n \geq 1.
  $$
  An application of the Burkholder--Davis--Gundy inequality and the Borel--Cantelli lemma, as in the proof of \cite[Proposition~3.4]{Liu2018}, then yields the existence of a measurable set $\Omega^\prime \subseteq \Omega$ with full measure such that, for every $\omega \in \Omega^\prime$ and every $\varepsilon \in (0,1)$, there exists a constant $C = C(\varepsilon,\omega)$ such that
  \begin{equation}\label{eq: uniform convergence of Ito integrals}
    \bigg\|\bigg(\int_0^{\cdot} X^n_{u-} \otimes \d X_u - \int_0^{\cdot} X_{u-} \otimes \d X_u \bigg)(\omega)\bigg\|_{\infty,[0,T]} \leq C 2^{-n(1-\varepsilon)}, \qquad \forall n \geq 1.
  \end{equation}
  Thus, we have that $X^n(\omega) \to X(\omega)$ and $\int_0^{\cdot} X^n_{u-} \otimes \d X_u(\omega) \to \int_0^{\cdot} X_{u-} \otimes \d X_u(\omega)$ uniformly as $n \to \infty$, for every $\omega \in \Omega'$.

  \smallskip

  \textit{Step 2.}
  We choose a $q_0 \in (2,3)$ close enough to $2$ and an $\varepsilon \in (0,1)$ small enough such that
  \begin{equation}\label{eq: choice of q and epsilon}
    \frac{p}{2}  > \max \Big\{\frac{q_0}{2}, q_0 -1\Big\} \qquad \text{and} \qquad \frac{p}{2} \geq \frac{q_0-1-\varepsilon}{1-\varepsilon}. 
  \end{equation}
  We also fix a control function $w_{X,q_0}$ such that
  \begin{equation}\label{eq:X control q_0}
    |X_{s,t}|^{q_0} \leq w_{X,q_0}(s,t) \qquad \text{for all} \quad (s,t) \in \Delta_{[0,T]}.
  \end{equation}
  This is always possible as $X$ has almost surely finite $q$-variation for every $q > 2$, and so without loss of generality we may assume that $X(\omega)$ has finite $q_0$-variation for every $\omega \in \Omega^\prime$. Note that by \eqref{eq: choice of q and epsilon} we have $p > q_0$ and consequently (by increasing the values of $w_{X,q_0}$ by a multiplicative constant if necessary) we can also assume that 
  \begin{equation}\label{eq: finite q variation of X}
    \sup_{(s,t) \in \Delta_{[0,T]}} \frac{|X_{s,t}|^p}{w_{X,q_0}(s,t)} \leq 1.
  \end{equation}
 
  Let $0 \leq s < t \leq T$ be such that $s = \tau^n_{k_0}$ and $t = \tau^n_{k_0 + N}$ for some $n \in \N$, $k_0 \in \N \cup \{0\}$ and $N \geq 1$. By the superadditivity of the control function $w_{X,q_0}$, there must exist an $l \in \{1,2,\ldots,N-1\}$ such that
  $$
    w_{X,q_0}(\tau^n_{k_0 + l - 1},\tau^n_{k_0 + l + 1}) \leq \frac{2}{N-1} w_{X,q_0}(s,t).
  $$
  Thus, by \eqref{eq:X control q_0},
  \begin{align*}
    |&X_{\tau^n_{k_0 + l - 1}} \otimes X_{\tau^n_{k_0 + l - 1},\tau^n_{k_0 + l}} + X_{\tau^n_{k_0 + l}} \otimes X_{\tau^n_{k_0 + l},\tau^n_{k_0 + l + 1}} - X_{\tau^n_{k_0 + l - 1}} \otimes X_{\tau^n_{k_0 + l - 1},\tau^n_{k_0 + l + 1}}|\\
    &= |X_{\tau^n_{k_0 + l - 1},\tau^n_{k_0 + l}} \otimes X_{\tau^n_{k_0 + l},\tau^n_{k_0 + l + 1}}| \leq w_{X,q_0}(\tau^n_{k_0 + l - 1},\tau^n_{k_0 + l + 1})^{2/q_0}\\
    &\leq \bigg(\frac{2}{N-1} w_{X,q_0}(s,t)\bigg)^{2/q_0}.
  \end{align*}
  By successively removing in this manner all the intermediate points from the partition $\{s = \tau^n_{k_0}, \tau^n_{k_0 + 1}, \ldots, \tau^n_{k_0 + N} = t\}$, we obtain the estimate
  \begin{align*}
     \bigg|\sum_{i=0}^{N-1} X_{\tau^n_{k_0 + i}} \otimes X_{\tau^n_{k_0 + i},\tau^n_{k_0 + i + 1}} - X_s \otimes X_{s,t}\bigg| &\leq \sum_{j=2}^N \bigg(\frac{2}{j-1} w_{X,q_0}(s,t)\bigg)^{2/q_0}\\
     &\lesssim N^{1 - 2/q_0} w_{X,q_0}(s,t)^{2/q_0}.
  \end{align*}
  Since $w_{X,q_0}(s,t) \geq \sum_{i=0}^{N-1} w_{X,q_0}(\tau^n_{k_0 + i},\tau^n_{k_0 + i + 1}) \geq \sum_{i=0}^{N-1} |X_{\tau^n_{k_0 + i},\tau^n_{k_0 + i + 1}}|^{q_0} \geq N 2^{-n q_0}$, we have that $N \leq 2^{n q_0} w_{X,q_0}(s,t)$. Substituting this into the above, we have that
  \begin{equation}\label{eq:est X^n d X new}
    \bigg|\int_s^t X^n_{u-} \otimes \d X_u - X_s \otimes X_{s,t}\bigg| \lesssim 2^{n(q_0 - 2)} w_{X,q_0}(s,t),
  \end{equation}
  where here the discretized process $X^n$ is defined relative to the partition $\{\tau^n_0,\tau^n_1,\tau^n_2,\ldots\}$.

  \smallskip

  If, more generally, $0 \leq s < t \leq T$ are such that $s, t \in \cP^n_X$, then $s = \tau^{m_1}_{k_1}$ and $t = \tau^{m_2}_{k_2}$ for some $m_1, m_2 \leq n$ and $k_1, k_2 \in \N \cup \{0\}$. In this case, the number of partition points $N$ above satisfies $N \leq \sum_{m=1}^n 2^{m q_0} w_{X,q_0}(s,t) \lesssim 2^{n q_0} w_{X,q_0}(s,t)$, and we thus still obtain the same bound in \eqref{eq:est X^n d X new}. If we further allow the pair of times $s, t$ to include the times $\sigma^m_k$ for $m \leq n$ and $k \in \N \cup \{0\}$, then this at most doubles the total number of partition points $N$, so we can again obtain the same bound. Since the points $\rho^m_{k,i}$ lie inside the interval $[\sigma^m_k,\varsigma^m_k]$ on which $X$ is constant, it is clear for instance that $X_{\sigma^m_k,\rho^m_{k,i}} = 0$ and $\int_{\sigma^m_k}^{\rho^m_{k,i}} X^n_{u-} \otimes \d X_u = 0$ for every $i \in \N$, and it follows that these terms will not contribute anything to the bound above.

  Thus, the bound in \eqref{eq:est X^n d X new} actually holds for all $0 \leq s < t \leq T$ with $s, t \in \mathcal{Q}^n_X$, with $X^n$ defined relative to the partition $\mathcal{Q}^n_X$, as in \eqref{eq:defn X^n}.

  \smallskip

  \textit{Step 3.}
  We first consider the case that $w_{X,q_0}(s,t)^{\frac{2}{p(1-\varepsilon)}} \leq 2^{-n}$. In this case it follows from \eqref{eq: choice of q and epsilon} and \eqref{eq:est X^n d X new} that there exists a constant $C_1 = C_1(\omega, q_0, \varepsilon)$ such that
  \begin{equation}\label{eq: control for first case}
    \bigg|\int_s^t X^n_{u-}\otimes \d X_u - X_s \otimes X_{s,t}\bigg|^{\frac{p}{2}} \leq C_1 w_{X,q_0}(s,t).
  \end{equation}
  Now we consider the case that $w_{X,q_0}(s,t)^{\frac{2}{p(1-\varepsilon)}} \ge 2^{-n}$. Let $\mathbb{X}_{s,t} := \int_s^t X_{u-} \otimes \d X_u - X_s \otimes X_{s,t}$ be the second level component of the It{\^o} lift of $X$. By \cite[Proposition~3.4]{Liu2018}, we know that $\mathbb{X}$ possesses finite $\frac{p}{2}$-variation; that is, there exists a control function $w_{\mathbb{X},\frac{p}{2}}$ such that
$$\sup_{(s,t) \in \Delta_{[0,T]}} \frac{|\mathbb{X}_{s,t}|^{\frac{p}{2}}}{w_{\mathbb{X},\frac{p}{2}}(s,t)} \leq 1.$$
  Then, in view of~\eqref{eq: uniform convergence of Ito integrals}, we obtain
  \begin{align*}
    \bigg|\int_s^t X^n_{u-} \otimes \d X_u - X_{s} \otimes X_{s,t}\bigg| &\leq 2 \bigg\|\int_0^{\cdot} X^n_{u-} \otimes \d X_u - \int_0^{\cdot} X_{u-} \otimes \d X_u\bigg\|_{\infty,[0,T]} + |\mathbb{X}_{s,t}|\\
    &\leq C_2 \Big(2^{-n(1-\varepsilon)} + w_{\mathbb{X},\frac{p}{2}}(s,t)^{\frac{2}{p}}\Big)\\
    &\leq C_2 \Big(w_{X,q_0}(s,t)^{\frac{2}{p}} + w_{\mathbb{X},\frac{p}{2}}(s,t)^{\frac{2}{p}}\Big)
  \end{align*}
  for some constant $C_2 = C_2(\varepsilon, \omega)$, and hence
  \begin{equation}\label{eq: control for second case}
    \bigg|\int_s^t X^n_{u-}\otimes \d X_u - X_s\otimes X_{s,t}\bigg|^{\frac{p}{2}} \leq C_3 \Big(w_{X,q_0}(s,t) + w_{\mathbb{X},\frac{p}{2}}(s,t)\Big),
  \end{equation}
  where $C_3 = (2 C_2)^{\frac{p}{2}}$. Letting $\tilde w_{X,p}(s,t) := 2(1 + C_1 + C_3)(w_{X,q_0}(s,t) + w_{\mathbb{X},\frac{p}{2}}(s,t))$, and combining \eqref{eq: finite q variation of X}, \eqref{eq: control for first case} and \eqref{eq: control for second case}, we conclude that for every $\omega \in \Omega^\prime$,
$$\sup_{(s,t) \in \Delta_{[0,T]}} \frac{|X_{s,t}|^{p}}{\tilde{w}_{X,p}(s,t)} + \sup_{n \in \N} \sup_{\substack{(s,t) \in \Delta_{[0,T]}\\
s, t \in \mathcal{Q}^n_X}} \frac{|\int_s^t X^n_{u-}\otimes \d X_u - X_s \otimes X_{s,t}|^{\frac{p}{2}}}{\tilde{w}_{X,p}(s,t)} \leq 1,$$
from which Property~\textup{(RIE)} follows.
\end{proof}

\begin{remark}\label{remark: RIE lift is Ito lift}
The result of Proposition~\ref{prop: semimartingales satisfy RIE} implies that almost every sample path of a semimartingale $X$ satisfies \textup{(RIE)} with respect to a suitable sequence of partitions (which depends on the choice of sample path), and may therefore be canonically lifted to a rough path. However, we also see from the proof of Proposition~\ref{prop: semimartingales satisfy RIE} that the resulting rough path lift $(s,t) \mapsto \int_s^t X_{s,u} \otimes \dd X_u$ is nothing but the standard It\^o-rough path lift of $X$. Thus, the rough path lift itself does not actually depend on the choice of sequence of partitions. For (almost) every sample path, the sequence of partitions specified in Property \textup{(RIE)} is merely a choice of partitions along which the value of the Ito integral may be well approximated by left-point Riemann sums in a pathwise fashion.

This is analogous to F\"ollmer integration, in which the quadratic variation of a semimartingale is well-defined as a stochastic process, but to make sense of the quadratic variation in a pathwise sense requires a suitable choice of sequence of partitions, which depends on the sample path being considered. Indeed, given any continuous path, it is possible to find a sequence of partitions along which the quadratic variation of the path is equal to zero (see Freedman \cite[(70) Proposition]{Freedman1983}). Similarly, it is natural to allow the sequence of partitions specified in Property \textup{(RIE)} to depend on the path being considered, but in practice there typically exists a natural choice for this sequence which results in the desired rough path.
\end{remark}

\begin{remark}
Proposition~\ref{prop: semimartingales satisfy RIE} holds true even if the semimartingale $X$ takes values in an (infinite dimensional) Hilbert space $E$, as long as the norm on $E \otimes E$ is admissible in the sense of Lyons, Caruana and L\'evy \cite[Definition~1.25]{Lyons2007}. In particular, an extension of Proposition~\ref{prop: semimartingales satisfy RIE} to so-called piecewise semimartingales, which were introduced in Strong \cite[Definition~2.2]{Strong2014b} as generalized semimartingales with an image dimension evolving randomly in time, appears to be straightforward to implement. As discussed in Karatzas and Kim \cite[Remark~6.2 and Section~7]{Karatzas2020}, piecewise semimartingales provide a realistic framework to model so-called open markets, which are financial markets with an evolving number of traded assets.
\end{remark}

\subsection{Generalized semimartingales}

It is a well observed fact in the empirical literature, see e.g.~Lo \cite{Lo1991}, that price processes appear regularly in financial markets which are not semimartingales. Motivated by this fact, many researchers have proposed and investigated financial models based on fractional Brownian motions; see for instance Jarrow, Protter and Sayit \cite{Jarrow2009}, Cheridito \cite{Cheridito2003} or Bender \cite{Bender2012}. One example of such models are the so-called mixed Black--Scholes models. In these models the (one-dimensional) price process $S=(S_t)_{t\in [0,\infty)}$ is usually given by
\begin{equation}\label{eq:mixed Black-Scholes model}
  S_t := s_0 \exp(\sigma W_t + \eta Y_t + \nu t + \mu t^{2H}), \qquad t \in [0,\infty),
\end{equation}
for constants $s_0, \sigma, \eta > 0$ and $\nu, \mu \in \R$, where $W=(W_t)_{t\in[0,\infty)}$ is standard Brownian motion and $Y=(Y_t)_{t\in [0,\infty)}$ is a fractional Brownian motion with Hurst index $H \in (0,1)$. Multi-dimensional versions of the mixed Black--Scholes model~\eqref{eq:mixed Black-Scholes model} can be obtained by standard modifications. Notice that, while the price process $S$ as defined in \eqref{eq:mixed Black-Scholes model} is not a semimartingale if $H \neq 1/2$, the mixed Black--Scholes model~\eqref{eq:mixed Black-Scholes model} is still arbitrage-free when restricting the admissible trading strategies to classes of trading strategies which, roughly speaking, exclude continuous rebalancing of the positions in the underlying market, cf.~\cite{Jarrow2009,Cheridito2003,Bender2012}. In particular, the mixed Black--Scholes model \eqref{eq:mixed Black-Scholes model} is free of simple arbitrage opportunities if $H > 1/2$, as proven in \cite[Section~4.1]{Bender2012}.

\smallskip

In order to demonstrate that Property~\textup{(RIE)} is satisfied by various financial models based on fractional Brownian motion, we consider the following class of generalized semimartingales. On a filtered probability space $(\Omega, \mathcal{F}, (\mathcal{F}_t)_{t \in [0,\infty)},\P)$ with a complete right-continuous filtration $(\mathcal{F}_t)_{t \in [0,\infty)}$, let $Z$ be a $d$-dimensional process admitting the decomposition
\begin{equation*}
  Z_t = X_t + Y_t, \qquad t \in [0,\infty),
\end{equation*}
where $X$ is a semimartingale, and $Y$ is a c{\`a}dl{\`a}g adapted process with finite $q$-variation for some $q \in [1,2)$. Processes $Z$ of this form are sometimes called \emph{Young semimartingales}, and belong to the class of c{\`a}dl{\`a}g Dirichlet processes in the sense of F\"ollmer \cite{Follmer1981b}.

\smallskip

We introduce stopping times $(\tau^n_k)$, such that, for each $n \in \N$, $\tau^n_0 = 0$, and
$$\tau^n_k := \inf \{t > \tau^n_{k-1} : |X_t - X_{\tau^n_{k-1}}| \geq 2^{-n} \text{ or } |Y_t - Y_{\tau^n_{k-1}}| \geq 2^{-n}\}$$
for $k \in \N$, and set
$$\cP_Z^n := \{\tau^m_k : m \leq n, k \in \N \cup \{0\}\}.$$
As in the previous section, since we insist that the sequence of partitions in Property~\textup{(RIE)} has vanishing mesh size, we also define
\begin{equation*}
  \mathcal{Q}^n_Z := \cP_Z^n \cup \{\sigma^m_k : m \leq n, k \in \N \cup \{0\}\} \cup \{\rho^m_{k,i} : m \leq n, k \in \N \cup \{0\}, i \in \N\},
\end{equation*}
where the times $\sigma^m_k$ and $\rho^m_{k,i}$ are defined analogously as in Section~\ref{subsect: semimartingale}.

\begin{proposition}
  Let $Z = X + Y$ be a $d$-dimensional process such that $X$ is a c{\`a}dl{\`a}g semimartingale, and $Y$ is a c{\`a}dl{\`a}g process with finite $q$-variation for some $q \in [1,2)$. Then, for any $p \in (2,3)$ such that $1/p + 1/q > 1$, almost every sample path of $Z$ is a price path in the sense of Definition~\ref{defn price path}.
\end{proposition}

\begin{proof}
  \textit{Step 1.}
  It is sufficient to prove that almost all sample paths of $Z$ satisfy Property \textup{(RIE)} along $(\mathcal{Q}^n_Z([0,T]))_{n \in \N}$ for an arbitrary $T > 0$. To this end, for each $n \in \mathbb{N}$ we set
  $$
    Z^n_t := \sum_{[u,v] \in \mathcal{Q}^n_Z([0,T])} Z_{u} \1_{[u,v)}(t), \qquad t \in [0,T],
  $$
  and define $X^n$ and $Y^n$ in the same way with respect to the partition $\mathcal{Q}^n_Z([0,T])$.

  Let $X = M + A$ be a decomposition of the semimartingale $X$ such that $M$ is a locally square integrable martingale and $A$ is of bounded variation. By setting $B := A + Y$, we can write $Z = M + B$, where $B$ has finite $q$-variation. We define the It{\^o} integral of $Z$ with respect to itself by
  $$
    \int_0^t Z_{u-} \otimes \d Z_u := \int_0^t Z_{u-} \otimes \d M_u + \int_0^t Z_{u-} \otimes \d B_u,
  $$
  where the first integral on the right-hand side is an It{\^o} integral and the second one is interpreted as a Young integral, which exists since $Z$ has finite $p$-variation and $1/p + 1/q > 1$; see e.g.~Friz and Zhang \cite{Friz2018}. Then, since $\|Z^n_{\cdot -} - Z_{\cdot -}\|_{\infty,[0,T]} \leq 2^{1 - n}$, by the Burkholder--Davis--Gundy inequality and the Borel--Cantelli lemma, we deduce that for almost all $\omega$ and for every $\varepsilon \in (0,1)$, there exists a constant $C = C(\omega,\varepsilon)$ such that, for all $n \in \N$,
  \begin{equation}\label{eq:int Z dM bound}
    \bigg\|\int_0^\cdot Z^n_{u-} \otimes \d M_u - \int_0^\cdot Z_{u-} \otimes \d M_u\bigg\|_{\infty,[0,T]} \leq C 2^{-n(1-\varepsilon)},
  \end{equation}
  cf.~the proof of \cite[Proposition~3.4]{Liu2018}. By a standard bound for Young integrals (e.g.~\cite[Proposition~2.4]{Friz2018}), for every $t \in [0,T]$, we also have (noting that $Z^n_0 = Z_0$ for all $n$) that
  $$
    \bigg|\int_0^t (Z^n_{u-} - Z_{u-}) \otimes \d B_u\bigg| \leq C \|Z^n_{\cdot -} - Z_{\cdot -}\|_{p,[0,t]}\|B\|_{q,[0,T]},
  $$
  for some constant $C = C(p,q)$. Since $\|Z^n_{\cdot -}\|_{p_0,[0,t]} \leq \|Z\|_{p_0,[0,T]}$ holds for every $n$ and every $p_0 \in (2,p)$, a routine interpolation argument shows that, for each $n \in \N$,
  $$
    \|Z^n_{\cdot -} - Z_{\cdot -}\|_{p,[0,T]} \leq C\|Z^n_{\cdot -} - Z_{\cdot -}\|_{\infty,[0,T]}^{1-\frac{p_0}{p}} \|Z\|_{p_0,[0,T]}^{\frac{p_0}{p}}
  $$
  for some constant $C = C(p,p_0)$. Hence, since $\|Z^n_{\cdot -} - Z_{\cdot -}\|_{\infty,[0,T]} \leq 2^{1-n}$ for all $n$, we have
  \begin{equation}\label{eq:int Z dB bound}
    \lim_{n \rightarrow \infty}\bigg\|\int_0^\cdot Z^n_{u-} \otimes \d B_u - \int_0^\cdot Z_{u-} \otimes \d B_u\bigg\|_{\infty,[0,T]} = 0.
  \end{equation}
  Combining \eqref{eq:int Z dM bound} with \eqref{eq:int Z dB bound}, we conclude that, almost surely, the integral $\int_0^\cdot Z^n_{u-} \otimes \d Z_u$ converges uniformly to $\int_0^\cdot Z_{u-} \otimes \d Z_u$.

  \smallskip

  \textit{Step 2.} For every $n\in \mathbb{N}$ we set
  \begin{align*}
     \mathbb{Z}^n_{s,t} &:= \int_s^t Z^n_{s,u-} \otimes \d Z_u\\
     &= \int_s^t X^n_{s,u-} \otimes \d X_u + \int_s^t Y^n_{s,u-} \otimes \d Y_u + \int_s^t X^n_{s,u-} \otimes \d Y_u + \int_s^t Y^n_{s,u-} \otimes \d X_u, 
  \end{align*}
  for $(s,t) \in \Delta_{[0,T]}$. Moreover, we define
  \begin{equation*}
     \mathbb{X}^n_{s,t} := \int_s^t X^n_{s,u-} \otimes \d X_u, \qquad \mathbb{Y}^n_{s,t}:= \int_s^t Y^n_{s,u-} \otimes \d Y_u,
  \end{equation*}
  and
  \begin{equation*}
    \mathbb{XY}^n_{s,t} := \int_s^t X^n_{s,u-} \otimes \d Y_u, \qquad \mathbb{YX}^n_{s,t} := \int_s^t Y^n_{s,u-} \otimes \d X_u,
  \end{equation*}
  for $(s,t) \in \Delta_{[0,T]}$. We seek a control function $c$ such that
  $$
    \sup_{n \in \N} \sup_{\substack{(s,t) \in \Delta_{[0,T]}\\
    s, t \in \mathcal{Q}^n_Z([0,T])}} \frac{|\mathbb{Z}^n_{s,t}|^{\frac{p}{2}}}{c(s,t)} \leq 1.
  $$
  Towards this aim, we first construct a control function $c_{\mathbb{X}}$ such that the above bound holds for $\mathbb{X}^n$. Since $\|X^n_{\cdot -} - X_{\cdot -}\|_{\infty,[0,T]} \leq 2^{1 - n}$, we still have the bound in \eqref{eq: uniform convergence of Ito integrals}. We also choose $q_0 \in (2,3)$ and $\varepsilon \in (0,1)$ as in \eqref{eq: choice of q and epsilon}, and let $w_{X,q_0}$ and $w_{Y,q}$ be control functions dominating the $q_0$-variation and $q$-variation for $X$ and $Y$ respectively. From the definition of the stopping times $\tau^m_k \in \cP^n_Z$ and the fact that $q < 2 < q_0$, it is easy to check that, for all $s < t$ with $s, t \in \cP^n_Z$, the number $N$ of partition points in $\cP^n_Z$ between $s$ and $t$ can be bounded by
  $$
    N \leq \sum_{m=1}^n \Big(2^{m q_0} w_{X,q_0}(s,t) + 2^{m q} w_{Y,q}(s,t)\Big) \lesssim 2^{n q_0} w_{q_0,q}(s,t),
  $$
  where $w_{q_0,q}(s,t) := w_{X,q_0}(s,t) + w_{Y,q}(s,t)$. By the same argument as in Step~2 of the proof of Proposition~\ref{prop: semimartingales satisfy RIE}, we deduce that, for all $n \in N$ and all $s < t$ with $s, t \in \mathcal{Q}^n_Z$,
  $$
    \bigg|\int_s^t X^n_{u-} \otimes \d X_u - X_s \otimes X_{s,t}\bigg| \lesssim 2^{n(q_0-2)} w_{q_0,q}(s,t),
  $$
  which, as in Step~3 of the proof of Proposition~\ref{prop: semimartingales satisfy RIE}, allows us to conclude the existence of a control function $c_{\mathbb{X}}$ such that
  $$
    \sup_{n \in \N} \sup_{\substack{(s,t) \in \Delta_{[0,T]}\\
    s, t \in \mathcal{Q}_Z^n([0,T])}} \frac{|\mathbb{X}^n_{s,t}|^{\frac{p}{2}}}{c_{\X}(s,t)} \leq 1.
  $$
  Next we use the local estimates of Young integration to show that there exists a control $c_{\mathbb{XY}}$ such that the above bound holds for $\mathbb{XY}^n$ and $c_{\mathbb{XY}}$. Indeed, by \cite[Proposition~2.4]{Friz2018}, for all $s<t$ in $[0,T]$, we have
\begin{align*}
\bigg|\int_s^t X^n_{s,u-} \otimes \d Y_u\bigg|^{p/2} &\leq C_{p,q} \|Y\|_{q,[s,t]|}^{p/2} \|X^n\|_{p,[s,t]}^{p/2} \leq C_{p,q} \|Y\|_{q,[s,t]}^{p/2} \|X\|_{p,[s,t]}^{p/2}\\
&\leq C_{p,q} w_{Y,q}(s,t)^{p/2q} w_{X,p}(s,t)^{1/2},
\end{align*}
  for some constant $C_{p,q}$ depending only on $p$ and $q$. Since $p \in (2,3)$ and $q \in [1,2)$, we have that $p/2q > 1/2$, and thus $p/2q + 1/2 > 1$, which implies that the map $(s,t) \mapsto w_{Y,q}(s,t)^{p/2q} w_{X,p}(s,t)^{1/2}$ is superadditive, and hence is itself a control function. Thus, the control function $c_{\mathbb{XY}} := C_{p,q} w_{Y,q}(s,t)^{p/2q} w_{X,p}(s,t)^{1/2}$ gives the desired bound. Similarly, we can find control functions $c_{\mathbb{YX}}$ and $c_{\mathbb{Y}}$ for $\mathbb{YX}^n$ and $\mathbb{Y}^n$, respectively. Hence, our claim follows by noting that
  $$
    |\mathbb{Z}^n| \leq |\mathbb{X}^n| + |\mathbb{Y}^n| + |\mathbb{XY}^n| + |\mathbb{YX}^n|.
  $$
  All together, we deduce that the sample paths of $Z$ almost surely satisfy Property~\textup{(RIE)} along $(\mathcal{Q}^n_Z([0,T]))_{n \in \N}$.
\end{proof}

\subsection{Typical price paths}

The notion of ``typical price paths'' was introduced by Vovk, who introduced a model-free hedging-based approach to mathematical finance allowing to investigate the sample path properties of such ``typical price paths'' based on arbitrage considerations; see for instance Vovk \cite{Vovk2008,Vovk2012} or Perkowski and Pr\"omel \cite{Perkowski2016}. Let us briefly recall the basic setting and definitions of Vovk's approach.

\smallskip

Let $\Omega_+:=D([0,\infty);\R^d_+)$ be the space of all non-negative c\`adl\`ag functions $\omega\colon [0,\infty)\to\R^d_+$. For each $t\in [0,\infty)$, ${\mathcal{F}}_{t}^{\circ}$ is defined to be the smallest $\sigma$-algebra on~$\Omega_+$ that makes all functions $\omega\mapsto\omega(s)$, $s\in[0,t]$, measurable and ${\mathcal{F}}_{t}$ is defined to be the universal completion of ${\mathcal{F}}_{t}^{\circ}$. Stopping times $\tau\colon\Omega_+\to [0,\infty)\cup \{ \infty \} $ with respect to the filtration $({\mathcal{F}}_{t})_{t\in[0,\infty)}$ and the corresponding $\sigma$-algebras ${\mathcal{F}}_{\tau}$ are defined as usual. The coordinate process on $\Omega_+$ is denoted by~$S$, i.e.~$S_{t}(\omega):=\omega(t)$ for $t\in [0,\infty)$.

\smallskip

A process $H\colon \Omega_+ \times [0,\infty)\to\R^d$ is a \emph{simple (trading) strategy} if there exists a sequence of stopping times $0 = \sigma_0 < \sigma_1 <  \sigma_2 < \dots$ such that for every $\omega\in\Omega_+$ there exists an $N(\omega)\in \N$ such that $\sigma_{n}(\omega)=\sigma_{n+1}(\omega)$ for all $n\geq N(\omega)$, and a sequence of $\mathcal{F}_{\sigma_n}$-measurable bounded functions $h_n\colon \Omega_+\to \R^d$, such that $H_t(\omega) = \sum_{n=0}^\infty h_n(\omega) \1_{(\sigma_n(\omega),\sigma_{n+1}(\omega)]}(t)$ for $t \in [0,\infty)$. For a simple strategy $H$, the corresponding integral process
\begin{equation*}
  (H \cdot S)_t(\omega) :=  \sum_{n=0}^\infty h_n(\omega) S_{\sigma_n\wedge t, \sigma_{n+1} \wedge t}(\omega) 
\end{equation*}
is well-defined for all $(t,\omega) \in [0,\infty)\times \Omega_+$. For $\lambda > 0$, we write $\mathcal{H}_{\lambda}$ for the set of all simple strategies $H$ such that $(H\cdot S)_t(\omega) \geq - \lambda $ for all $(t,\omega) \in [0,\infty)\times \Omega_+$.

\begin{definition}
  \emph{Vovk's outer measure} $\overline{P}$ of a set $A \subseteq \Omega_+$ is defined as the minimal super-hedging price for $\1_A$, that is
  \begin{equation*}
    \overline{P}(A) := \inf\Big\{\lambda > 0 : \exists (H^n)_{n \in \N} \subset \mathcal{H}_{\lambda} \text{ s.t. } \forall \omega \in \Omega_+, \liminf_{n \to \infty} (\lambda + (H^n\cdot S)_T(\omega)) \ge \1_A(\omega)\Big\}.
  \end{equation*}
  A given set $A \subseteq \Omega_+$ is called a \emph{null set} if it has outer measure zero. A property (P) holds for \emph{typical price paths} if the set $A$ where (P) is violated is a null set.
\end{definition} 

\begin{remark}
  Loosely speaking, the outer measure~$\overline{P}$ corresponds to the (model-free)  notion of ``no unbounded profit with bounded risk'', see \cite[Section~2.2]{Perkowski2016} for a more detailed discussion in this direction. Furthermore, the outer measure~$\overline{P}$ dominates all local martingale measures on the space~$\Omega_+$, see \cite[Lemma~2.3 and Proposition~2.5]{Lochowski2018}. As a consequence, all results proven for typical price paths hold simultaneously under all martingale measures (or in other words quasi-surely with respect to all martingale measures).
\end{remark}

Let us recall that typical price paths are of finite $p$-variation for every $p > 2$ (see Vovk \cite[Theorem~1]{Vovk2011}) and Vovk's model-free framework allows for setting up a model-free It{\^o} integration, see e.g.~\L{}ochowski, Perkowski and Pr\"omel \cite{Lochowski2018}.

\begin{lemma}\label{lem: typical price paths is RIE}
Typical price paths are price paths in the sense of Definition~\ref{defn price path}, with any $p \in (2,3)$.
\end{lemma}

The proof of Lemma~\ref{lem: typical price paths is RIE} works verbatim as that of Proposition~\ref{prop: semimartingales satisfy RIE}, keeping in mind \cite[Proposition~3.10]{Liu2018} and \cite[Corollary~4.9]{Lochowski2018}, and is therefore omitted for brevity.

\subsection{Consistency of rough and stochastic integration}\label{subsect: rough integral vs stochastic integral}

In a probabilistic framework when the underlying process is a semimartingale, one can employ either rough or stochastic It\^o integration. In this subsection we briefly demonstrate that, under Property \textup{(RIE)}, these two integrals actually coincide almost surely whenever both are defined.

\smallskip

Let us again fix a filtered probability space $(\Omega,\mathcal{F},(\mathcal{F}_t)_{t \in [0,\infty)},\P)$, and assume that the filtration $(\mathcal{F}_t)_{t \in [0,\infty)}$ satisfies the usual conditions.

\begin{proposition}
Let $X = (X_t)_{t \in [0,\infty)}$ be a $d$-dimensional c{\`a}dl{\`a}g semimartingale. Let $Y$ be an adapted c\`adl\`ag process such that, for almost every $\omega \in \Omega$, the path $Y(\omega) \in \mathcal{A}_{X(\omega)}$ is an admissible strategy (in the sense of Definition~\ref{def: admissible strategy}). Then the rough and It\^o integrals of $Y$ against $X$ coincide almost surely. That is,
$$\int_0^t Y_s(\omega) \dd \bX_s(\omega) = \bigg(\int_0^t Y_{s-} \dd X_s\bigg)(\omega), \qquad \text{for all} \quad t \in [0,\infty),$$
for almost every $\omega \in \Omega$, where $\bX(\omega)$ is the canonical rough path lift of $X(\omega)$, as defined in Lemma~\ref{lem: SSA is rough path}.
\end{proposition}

\begin{proof}
Fix $T > 0$. By Proposition~\ref{prop: semimartingales satisfy RIE}, we know that, for any $p \in (2,3)$, almost every sample path of $X$ is a price path, and satisfies Property \textup{(RIE)} along a nested sequence of adapted partitions $\mathcal{P}^n_X = \{0 = t^n_0 < t^n_1 < \cdots < t^n_{N_n} = T\}$, $n \in \N$, of the interval $[0,T]$ with vanishing mesh size. Since $Y(\omega) \in \mathcal{A}_{X(\omega)}$, there exists a c\`adl\`ag process $Y'$, such that $(Y(\omega),Y'(\omega)) \in \mathcal{V}^{q,r}_{X(\omega)}$ is a controlled path on $[0,T]$, for almost every $\omega \in \Omega$ and some suitable numbers $q, r$.

By \cite[Theorem~II.21]{Protter2005}, we have that
\begin{equation}\label{eq:Riemann sums converge to Ito int in prob}
\sum_{k=0}^{N_n-1} Y_{t^n_k} X_{t^n_k \wedge t,t^n_{k+1} \wedge t} \longrightarrow \int_0^t Y_{s-} \dd X_s \qquad \text{as} \quad n \longrightarrow \infty,
\end{equation}
uniformly in probability, for $t \in [0,T]$. By taking a subsequence if necessary, we can then assume that the (uniform) convergence in \eqref{eq:Riemann sums converge to Ito int in prob} holds almost surely. On the other hand, by Theorem~\ref{thm: rough int as limit Riemann sums}, we know that, for almost every $\omega \in \Omega$,
\begin{equation}\label{eq:Riemann sums conv to rough int in prob setting}
\sum_{k=0}^{N_n-1} Y_{t^n_k}(\omega) X_{t^n_k \wedge t,t^n_{k+1} \wedge t}(\omega) \longrightarrow \int_0^t Y_s(\omega) \dd \bX_s(\omega) \qquad \text{as} \quad n \longrightarrow \infty,
\end{equation}
uniformly for $t \in [0,T]$. Combining \eqref{eq:Riemann sums converge to Ito int in prob} and \eqref{eq:Riemann sums conv to rough int in prob setting}, we deduce that, almost surely, $\int_0^t Y_s \dd \bX_s = \int_0^t Y_{s-} \dd X_s$ for all $t \in [0,T]$. Since $T > 0$ was arbitrary, the result follows.
\end{proof}

\bibliography{quellen}{}

\def\cprime{$'$}
\begin{thebibliography}{10}

\bibitem{Allan2023}
A.~L. Allan, C.~Cuchiero, C.~Liu, and D.~J. Pr{\"o}mel.
\newblock Model-free portfolio theory: {A} rough path approach.
\newblock {\em Math. Finance}, 2023.
\newblock to appear.

\bibitem{Ananova2019}
A.~Ananova.
\newblock {\em Pathwise Integration and functional calculus for paths with
  finite quadratic variation}.
\newblock PhD thesis, https://doi.org/10.25560/66091, Imperial College London,
  2019.

\bibitem{Ananova2020}
A.~Ananova.
\newblock Rough differential equations with path-dependent coefficients.
\newblock {\em Preprint arXiv:2001.10688}, 2020.

\bibitem{Armstrong2018}
J.~Armstrong, C.~Bellani, D.~Brigo, and T.~Cass.
\newblock Option pricing models without probability: {A} rough paths approach.
\newblock {\em Math. Finance}, 31(4):1494--1521, 2021.

\bibitem{Avellaneda1995}
M.~Avellaneda, A.~Levy, and A.~Par\'{a}s.
\newblock {Pricing and hedging derivative securities in markets with uncertain
  volatilities}.
\newblock {\em Appl. Math. Finance}, 2(2):73--88, 1995.

\bibitem{Bender2012}
C.~Bender.
\newblock Simple arbitrage.
\newblock {\em Ann. Appl. Probab.}, 22(5):2067--2085, 2012.

\bibitem{Bichteler1981}
K.~Bichteler.
\newblock Stochastic integration and {$L^{p}$}-theory of semimartingales.
\newblock {\em Ann. Probab.}, 9(1):49--89, 1981.

\bibitem{Cheridito2003}
P.~Cheridito.
\newblock Arbitrage in fractional {B}rownian motion models.
\newblock {\em Finance Stoch.}, 7(4):533--553, 2003.

\bibitem{Chevyrev2019}
I.~Chevyrev and P.~K. Friz.
\newblock {Canonical RDEs and general semimartingales as rough paths}.
\newblock {\em Ann. Probab.}, 47(1):420--463, 01 2019.

\bibitem{Chiu2022}
H.~Chiu and R.~Cont.
\newblock Causal functional calculus.
\newblock {\em Trans. London Math. Soc.}, 9(1):237--269, 2022.

\bibitem{Cont2010}
R.~Cont and D.-A. Fourni{\'e}.
\newblock {Change of variable formulas for non-anticipative functionals on path
  space}.
\newblock {\em J. Funct. Anal.}, 259(4):1043--1072, 2010.

\bibitem{Cover1991}
T.~M. Cover.
\newblock Universal portfolios.
\newblock {\em Math. Finance}, 1(1):1--29, 1991.

\bibitem{Cuchiero2019}
C.~Cuchiero, W.~Schachermayer, and T.-K.~L. Wong.
\newblock Cover's universal portfolio, stochastic portfolio theory, and the
  numéraire portfolio.
\newblock {\em Math. Finance}, 29(3):773--803, 2019.

\bibitem{Davis2014}
M.~Davis, J.~Ob\l{\'o}j, and V.~Raval.
\newblock Arbitrage bounds for prices of weighted variance swaps.
\newblock {\em Math. Finance}, 24(4):821--854, 2014.

\bibitem{Delbaen1994}
F.~Delbaen and W.~Schachermayer.
\newblock {A general version of the fundamental theorem of asset pricing}.
\newblock {\em Math. Ann.}, 300(1):463--520, 1994.

\bibitem{Dolinsky2014}
Y.~Dolinsky and H.~M. Soner.
\newblock Martingale optimal transport and robust hedging in continuous time.
\newblock {\em Probab. Theory Related Fields}, 160(1-2):391--427, 2014.

\bibitem{Dupire2019}
B.~Dupire.
\newblock Functional {I}t{\^o} calculus.
\newblock {\em Quant. Finance}, 19(5):721--729, 2019.

\bibitem{Fernholz2002}
E.~R. Fernholz.
\newblock {\em Stochastic Portfolio Theory}.
\newblock Springer, 2002.

\bibitem{Follmer1981}
H.~F{\"o}llmer.
\newblock Calcul d'{I}t\^{o} sans probabilit\'{e}s.
\newblock In {\em Seminar on {P}robability, {XV} ({U}niv. {S}trasbourg,
  {S}trasbourg, 1979/1980)}, volume 850 of {\em Lecture Notes in Math.}, pages
  143--150. Springer, Berlin, 1981.

\bibitem{Follmer1981b}
H.~F{\"o}llmer.
\newblock Dirichlet processes.
\newblock In {\em Stochastic integrals ({P}roc. {S}ympos., {U}niv. {D}urham,
  {D}urham, 1980)}, volume 851 of {\em Lecture Notes in Math.}, pages 476--478.
  Springer, Berlin, 1981.

\bibitem{Follmer2013}
H.~F{\"o}llmer and A.~Schied.
\newblock Probabilistic aspects of finance.
\newblock {\em Bernoulli}, 19(4):1306--1326, 2013.

\bibitem{Frankova2019}
D.~Fra\v{n}kov\'{a}.
\newblock Regulated functions with values in {B}anach space.
\newblock {\em Math. Bohem.}, 144(4):437--456, 2019.

\bibitem{Freedman1983}
D.~Freedman.
\newblock {\em Brownian motion and diffusion}.
\newblock Springer-Verlag, 1983.

\bibitem{FrizHairer2020}
P.~K. Friz and M.~Hairer.
\newblock {\em A Course on Rough Paths with an Introduction to Regularity
  Structures}.
\newblock Springer, 2nd edition, 2020.

\bibitem{Friz2017}
P.~K. Friz and A.~Shekhar.
\newblock General rough integration, {L}\'{e}vy rough paths and a
  {L}\'{e}vy-{K}intchine-type formula.
\newblock {\em Ann. Probab.}, 45(4):2707--2765, 2017.

\bibitem{Friz2018}
P.~K. Friz and H.~Zhang.
\newblock Differential equations driven by rough paths with jumps.
\newblock {\em J. Differential Equations}, 264(10):6226--6301, 2018.

\bibitem{Hobson2003}
D.~Hobson.
\newblock The {S}korokhod embedding problem and model-independent bounds for
  option prices.
\newblock In {\em Paris-{P}rinceton {L}ectures on {M}athematical {F}inance
  2010}, volume 2003 of {\em Lecture Notes in Math.}, pages 267--318. Springer,
  Berlin, 2011.

\bibitem{Hou2018}
Z.~Hou and J.~Ob\l{\'o}j.
\newblock Robust pricing-hedging dualities in continuous time.
\newblock {\em Finance Stoch.}, 22(3):511--567, 2018.

\bibitem{Jarrow2009}
R.~A. Jarrow, P.~Protter, and H.~Sayit.
\newblock No arbitrage without semimartingales.
\newblock {\em Ann. Appl. Probab.}, 19(2):596--616, 2009.

\bibitem{Karandikar1995}
R.~L. Karandikar.
\newblock {On pathwise stochastic integration}.
\newblock {\em Stochastic Process. Appl.}, 57(1):11--18, may 1995.

\bibitem{Karatzas2007}
I.~Karatzas and C.~Kardaras.
\newblock {The num\'{e}raire portfolio in semimartingale financial models}.
\newblock {\em Finance Stoch.}, 11(4):447--493, 2007.

\bibitem{Karatzas2020}
I.~Karatzas and D.~Kim.
\newblock Trading strategies generated pathwise by functions of market weights.
\newblock {\em Finance Stoch.}, 24(2):423--463, 2020.

\bibitem{Karatzas2017}
I.~Karatzas and J.~Ruf.
\newblock Trading strategies generated by {L}yapunov functions.
\newblock {\em Finance Stoch.}, 21(3):753--787, 2017.

\bibitem{Liu2018}
C.~Liu and D.~J. Pr\"{o}mel.
\newblock Examples of {I}t\^{o} c\`adl\`ag rough paths.
\newblock {\em Proc. Amer. Math. Soc.}, 146(11):4937--4950, 2018.

\bibitem{Lo1991}
A.~W. Lo.
\newblock Long-term memory in stock market prices.
\newblock {\em Econometrica}, 59(5):1279--1313, 1991.

\bibitem{Lochowski2018}
R.~M. {\L}ochowski, N.~Perkowski, and D.~J. Pr{\"o}mel.
\newblock A superhedging approach to stochastic integration.
\newblock {\em Stochastic Process. Appl.}, 128(12):4078--4103, 2018.

\bibitem{Lyons1995}
T.~J. Lyons.
\newblock {Uncertain volatility and the risk-free synthesis of derivatives}.
\newblock {\em Appl. Math. Finance}, 2(2):117--133, 1995.

\bibitem{Lyons1998}
T.~J. Lyons.
\newblock {Differential equations driven by rough signals}.
\newblock {\em Rev. Mat. Iberoam.}, 14(2):215--310, 1998.

\bibitem{Lyons2007}
T.~J. Lyons, M.~Caruana, and T.~L\'{e}vy.
\newblock {\em Differential equations driven by rough paths}, volume 1908 of
  {\em {Lecture Notes in Mathematics}}.
\newblock Springer, Berlin, 2007.

\bibitem{Nutz2012}
M.~Nutz.
\newblock Pathwise construction of stochastic integrals.
\newblock {\em Electron. Commun. Probab.}, 17(24):1--7, 2012.

\bibitem{Perkowski2016}
N.~Perkowski and D.~J. Pr\"{o}mel.
\newblock Pathwise stochastic integrals for model free finance.
\newblock {\em Bernoulli}, 22(4):2486--2520, 2016.

\bibitem{Protter2005}
P.~E. Protter.
\newblock {\em Stochastic integration and differential equations}.
\newblock Springer, 2nd edition, 2005.

\bibitem{Riga2016}
C.~Riga.
\newblock A pathwise approach to continuous-time trading.
\newblock {\em Preprint arXiv:1602.04946}, 2016.

\bibitem{Schied2018}
A.~Schied, L.~Speiser, and I.~Voloshchenko.
\newblock Model-free portfolio theory and its functional master formula.
\newblock {\em SIAM J. Financial Math.}, 9(3):1074--1101, 2018.

\bibitem{Schied2016}
A.~Schied and I.~Voloshchenko.
\newblock Pathwise no-arbitrage in a class of delta hedging strategies.
\newblock {\em Probab. Uncertain. Quant. Risk}, 1:Paper No. 3, 25, 2016.

\bibitem{Strong2014b}
W.~Strong.
\newblock Fundamental theorems of asset pricing for piecewise semimartingales
  of stochastic dimension.
\newblock {\em Finance Stoch.}, 18(3):487--514, 2014.

\bibitem{Strong2014}
W.~Strong.
\newblock Generalizations of functionally generated portfolios with
  applications to statistical arbitrage.
\newblock {\em SIAM J. Financial Math.}, 5(1):472--492, 2014.

\bibitem{Vovk2008}
V.~Vovk.
\newblock {Continuous-time trading and the emergence of volatility}.
\newblock {\em Electron. Commun. Probab.}, 13:319--324, 2008.

\bibitem{Vovk2011}
V.~Vovk.
\newblock {Rough paths in idealized financial markets}.
\newblock {\em Lith. Math. J.}, 51(2):274--285, 2011.

\bibitem{Vovk2012}
V.~Vovk.
\newblock {Continuous-time trading and the emergence of probability}.
\newblock {\em Finance Stoch.}, 16(4):561--609, 2012.

\bibitem{Vovk2015}
V.~Vovk.
\newblock It{\^o} calculus without probability in idealized financial markets.
\newblock {\em Lith. Math. J.}, 55(2):270--290, 2015.

\bibitem{Willinger1989}
W.~Willinger and M.~S. Taqqu.
\newblock Pathwise stochastic integration and applications to the theory of
  continuous trading.
\newblock {\em Stochastic Process. Appl.}, 32(2):253--280, 1989.

\end{thebibliography}
\bibliographystyle{abbrv}

\end{document}